\documentclass[twoside]{article}
\usepackage{setspace}
\usepackage{amsfonts, amsthm, amsmath, amssymb}
\usepackage{ulem}
\usepackage{float}
\usepackage{comment}
\usepackage{graphicx}
\usepackage[caption = false]{subfig}
\usepackage{epstopdf}
\usepackage[ruled,vlined,linesnumbered]{algorithm2e}
\usepackage{setspace}
\usepackage[shortlabels]{enumitem}
\usepackage{amsopn}
\usepackage{bbm}
\usepackage{mathtools}
\usepackage{array,multirow,booktabs}
\usepackage{hyperref}
\usepackage{cleveref}
\usepackage{xcolor}
\usepackage{listings}
\usepackage{lipsum}
\usepackage[top=1in, bottom=1.25in, left=1.25in, right=1.25in]{geometry}
\numberwithin{equation}{section}
\ifpdf
  \DeclareGraphicsExtensions{.eps,.pdf,.png,.jpg}
\else
  \DeclareGraphicsExtensions{.eps}
\fi
\hypersetup{
    colorlinks=true,
    linkcolor=blue,
    filecolor=magenta,      
    urlcolor=cyan,
}
\newcommand{\R}{\mathbbm{R}}

\newcommand{\an}[1]{{\leavevmode\color{red}{\bf #1}}}

\newcommand{\email}[1]{\protect\href{mailto:#1}{#1}}
\newcommand{\bhv}{\boldsymbol{\widehat{v}}}

\newcommand{\bhp}{\boldsymbol{\widehat{p}}}
\newcommand{\bhl}{\boldsymbol{\widehat{\ell}}}

\newcommand{\btp}{\boldsymbol{\widetilde{p}}}
\newcommand{\bc}{\boldsymbol{c}}

\newcommand{\bhc}{\boldsymbol{\widehat{c}}}
\newcommand{\bx}{\boldsymbol{x}}
\newcommand{\bq}{\boldsymbol{q}}

\newcommand{\bA}{\boldsymbol{A}}
\newcommand{\KC}{K_C}
\newcommand{\bs}[1]{\boldsymbol{#1}}
\newcommand{\bd}{\boldsymbol{d}}

\newcommand{\vlc}{v_{\text{LC}}}
\newcommand{\vnc}{v_{\text{NC}}}
\newcommand{\tarc}{\mbox{\large$\frown$}}
\newcommand{\arc}[2][-3ex]{{#2}{\kern #1{\raisebox{1.5ex}{\tarc}}}}
\newtheorem{theorem}{Theorem}[section]
\newtheorem{remark}{Remark}[section]
\newtheorem{corollary}{Corollary}[section]
\newtheorem{lemma}{Lemma}[section]
\newtheorem{proposition}{Proposition}[section]
\theoremstyle{definition}
\newtheorem{definition}{Definition}[section]
\newtheorem{example}{Example}[section]
\newtheorem{assumption}{Assumption}[section]
\DeclareMathOperator*{\argmax}{arg\,max\;}
\DeclareMathOperator*{\argmin}{arg\,min\;}
\title{Non-Dissipative and Structure-Preserving Emulators via Spherical Optimization}
\author{Dihan Dai\thanks{Department of Mathematics, University of Utah, Salt Lake City, UT 84112 (\email{dai@math.utah.edu}, \email{epshteyn@math.utah.edu}).} \thanks{Scientific Computing and Imaging (SCI) Institute, University of Utah, Salt Lake City, UT 84112 (\email{akil@sci.utah.edu}).
D. Dai and A. Narayan are supported by NSF DMS-1848508 and AFOSR FA9550-20-1-0338.
}
\and Yekaterina Epshteyn \footnotemark[1]
\and Akil Narayan\footnotemark[1] \footnotemark[2]}
\begin{document}
\maketitle
\begin{abstract}
  Approximating a function with a finite series, e.g., involving polynomials or trigonometric functions, is a critical tool in computing and data analysis. The construction of such approximations via now-standard approaches like least squares or compressive sampling does not ensure that the approximation adheres to certain convex linear structural constraints, such as positivity or monotonicity. Existing approaches that ensure such structure are norm-dissipative and this can have a deleterious impact when applying these approaches, e.g., when numerical solving partial differential equations.  We present a new framework that enforces via optimization such structure on approximations and is simultaneously norm-preserving. This results in a conceptually simple convex optimization problem on the sphere, but the feasible set for such problems can be very complex. We establish well-posedness of the optimization problem through results on spherical convexity and design several spherical-projection-based algorithms to numerically compute the solution. Finally, we demonstrate the effectiveness of this approach through several numerical examples. 
%    When an unknown function is approximated using a finite collection of observations, its approximation often violates the underlying properties of the original function. Based on the linear structure-preserving framework introduced in [Zala, V., Kirby, M., and Narayan, A. \textit{SIAM Journal on Scientific Computing}, 42(5), A3006-A3029, 2020], we further add and study a quadratic energy-preserving structure. The resulting new optimization formulation is a spherically convex optimization problem on the sphere under reasonable assumptions. However, the feasible set for such problem can be very complex. Thus, we design and discuss several spherical-projection-based algorithms to numerically compute the solution, which can be viewed as the extensions of the methods for enforcing linear structures. The numerical results demonstrate the effectiveness of our algorithms.
    
\end{abstract}
\section{Introduction}
Approximating an unknown function with a superposition of basis functions (e.g., polynomials or Fourier series) is a widely-used technique in computing and numerical analysis. For example, when solving a system of partial differential equations (PDEs), the class of spectral methods proposes such a superposition ansatz and determines the coefficients through minimization conditions on the PDE residual. Traditionally, fundamental properties of the approximation, such as stability, accuracy, and computational efficiency are major considerations for the approximations. However, for certain problems, approximations are required to preserve certain implicit ``structures'', i.e., approximations should inherit certain desirable qualitative features of the original function. Such structure can include positivity, monotonicity, conservation of energy, etc. An approximation that fails to be structure-preserving may lead to numerical instability or even the failure of numerical schemes \cite{zhang2011maximum}. From the broader viewpoint of building predictive emulators from data, this structure can be crucial to generate a meaningful emulator; for example, emulators built to predict population trends should not predict negative values.  In this manuscript we consider building approximations that respect general families of linear homogeneous convex inequality constraints (for which positivity and monotonicity are examples) along with a single quadratic equality constraint (an energy constraint).

Based on the existing framework for linear inequality constraints \cite{zala2020structure}, we impose a new \textit{spherical} constraint, i.e., a quadratic constraint in addition to the linear constraints. While a seemingly benign addition, this extra constraint substantially changes the optimization problem and its properties. With a coordinate vector $\bhp$ provided (e.g., a vector of Fourier coefficients), the formulation we consider in this paper gives rise to an optimization problem of the form,
\begin{align}\label{eq:optgeneral}
    &\min_{\bhv\in\mathbb{R}^N}\Vert \bhv - \bhp\Vert_2^2\nonumber\\
    \text{s.t. }&g_k(\bhv, y)\le 0,\;\;\forall y\in \omega_k, k\in [K]\\
               &\Vert\bhv\Vert_2 = \Vert\bhp\Vert_2\nonumber,
\end{align}
where the linear constraints are given by the $y$-parameterized scalar-valued functions $g_k(\cdot, y)$ and the energy-preserving constraint is given by the equality constraint $\Vert\bhv\Vert_2 = \Vert\bhp\Vert_2$. The parameters $y$ can take values from a (possibly uncountably infinite) set $\omega_k$, and hence the feasible set can be very complex. Generally, the feasible set in \eqref{eq:optgeneral} is the intersection of a finite collection of homogeneous convex cones and a sphere. The model \eqref{eq:optgeneral} corresponds to a semi-infinite programming (SIP) problem \cite{hettich1993semi}. In several SIP algorithms, a discrete approximation to the domain $\omega_k$ is constructed (and perhaps refined). For linear constraints corresponding to positivity, this would correspond to requiring positivity at only a finite collection of points on the domain and hence structure is only preserved at a discrete set of points instead of on the whole domain. An additional difficulty is that the feasible set is a subset on the surface of a sphere, which is not a convex set in Euclidean space and hence the approaches from \cite{zala2020structure} do not apply.  Thus, computationally solving \eqref{eq:optgeneral} can be very challenging. 

\subsection{Related problems and approaches}
There is existing literature on the study of optimization over ellipsoids (or spheres), which is closely related to the solution of the subproblems in the class of trust-region methods \cite{gander1989constrained,hager2001minimizing,rendl1997semidefinite,sorensen1997minimization}. However, in those approaches, the number of linear constraints is finite, and therefore such approaches are not directly applicable in our setting.

%In this work, we develop a non-dissipative structure-preserving function approximation algorithm. 
There are existing energy-preserving numerical methods that focus on energy-conservation of a Hamiltonian system, where a differential equation is discretized in a special way so that the energy of the discretized system is preserved, see \cite{quispel2008new,celledoni2009energy,hairer2010energy,besse2021energy}. 
However, our focus is to preserve the energy of the approximation to a given function rather than the energy of a differential equation system, which is a different problem. In addition, methods built for differential equations assume very particular types of discretizations; the formulation we investigate in this paper applies to general discretizations.

A number of techniques have been proposed for preserving special kinds of structure for special choices of basis functions. To ensure positivity preservation, one can simply enforce positivity at a finite collection of points in the computational domain. The corresponding feasible set is a convex polytope and there are several algorithms available to computationally solve this problem \cite{boyd2004convex}. Unfortunately, such techniques do not guarantee positivity of the approximation over the entire domain (a generally uncountable Euclidean set). An alternative to constraints over a finite set is to use special mappings. For example, one can approximate $\sqrt{f}$ and square the resulting approximation, or approximate $\log f$ and subsequently exponentiate the approximation in order to guarantee the resulting approximation is positive. However, such mapping functions are not easy to construct for more complicated constraints, and the introduction of such maps can affect accuracy; for example, $x \mapsto \sqrt{x}$ is not smooth at $x = 0$. For univariate polynomial approximation, one can take advantage of the special representations of nonnegative polynomials (Luk\'{a}cs theorem) to develop more complicated iterative procedures \cite{doi:10.1137/17M1131891}. Another approach is to use an adaptive construction scheme for certain kinds of constraints \cite{berzins2007adaptive}. One can also linearly scale the high order coefficients of the polynomial to limit the oscillations \cite{zhang2011maximum}. Finally we note that there are existing theoretical investigations for structure-preserving approximation in \cite{devore1974degree,beatson1978degree,beatson1982restricted,nochetto2002positivity}, but these investigations do not translate into algorithms. 

Our approach extends the recent technique in \cite{zala2020structure}, which considers building approximations with \textit{linear} structure, in the sense that the constraints are linear with respect to the approximant. In \cite{zala2020structure}, the authors formalize a model for the structure-preserving problem with linear constraints, which applies to general, nontrivial linear structure. Under a mild condition, the corresponding function approximation problem can be cast to a semi-infinite convex optimization problem in a finite-dimensional Euclidean space with a unique solution. In addition, the work in \cite{zala2020structure} develops several projection-based algorithms to preserve the desired structures. However, their method, which amounts to filtering the approximation in a nonlinear manner, does not preserve the $L^2$ norm, and thus is dissipative. The work of this paper preserves the quadratic (energy) norm through a modified formulation of the problem. This slight modification results in nontrivial changes to well-posedness and algorithmic development that we address.

\subsection{Contributions of this paper}
In this work, we are interested in providing theory and algorithms to address non-dissipative, structure-preserving function approximation methods of the form \eqref{eq:optgeneral}. Using notions of spherical convexity and spherical projections \cite{ferreira2013projections,ferreira2014concepts}, we show that the corresponding function approximation problem can be converted to a spherically convex feasibility problem, and establish uniqueness of the solution under mild conditions, see \Cref{thm:uniqueness}. Based on our theoretical results and by extending algorithms in \cite{zala2020structure}, we propose three algorithms to solve the spherically convex feasibility problem; see sections \ref{ssec:greedy}, \ref{ssec:averaging}, and \ref{ssec:hybrid}. Our algorithms do not rely on the discretization of the domain and therefore differ from many existing SIP algorithms \cite{stein2012solve,goberna2013semi}.
 
The setup of the general problem is as follows: We first assume that the unconstrained approximation to the unknown function is available, e.g., an
unconstrained projection of the unknown function onto a finite-dimensional subspace. The unconstrained approximation is then \textit{post-processed} via our algorithms so that the linear constraints, such as positivity, are satisfied without augmenting or reducing the quadratic energy of the approximant.

This paper is structured as follows. In \Cref{sect:setup}, we give the theoretical framework of the structure-preserving function approximation problem as well as formalization of the constraints. In \Cref{sect:soln}, we provide a brief overview of spherical geometry and present the uniqueness result of the function approximation problem. In \Cref{sect:algorithms}, we discuss projections on the sphere and develop two algorithms for solving the function approximation problem. Finally, in \Cref{sect:numeric}, we demonstrate the efficacy of our algorithms with numerical results for polynomial and Fourier series approximations. Our energy- and structure-preserving results show similar rates of convergence as those of the unconstrained approximation as the subspace is refined. %solution and can be viewed as slight corrections to the constrained approximations obtained in \cite{zala2020structure}.
\section{Setup}\label{sect:setup}
Let $\Omega\subseteq\mathbb{R}^d$ be a spatial domain. Consider the Hilbert space $H$ formed by scalar-valued functions over $\Omega$ with inner product $\langle\cdot,\cdot\rangle_H$,
\begin{align*}
    &H = H(\Omega) \coloneqq \{f:\Omega\to \mathbb{R}\;\;|\;\;\Vert f\Vert_H<\infty\},&\Vert f\Vert^2\coloneqq \langle f,f\rangle_H,
\end{align*}
A prototypical example is $H = L^2(\Omega; \mathbb{R})$. Let $V\subseteq H$ be an $N$-dimensional subspace spanned by orthonormal basis functions $\{v_n\}_{n\in[N]}$,
\begin{align*}
  V &= \text{span}\{v_1,\cdots,v_N\},&\langle v_j,v_k\rangle_H &= \delta_{j,k}, &  j,k&\in [N],
\end{align*}
where $\delta_{j,k}$ is Kronecker delta function, and $[N] \coloneqq \{1,\cdots, N\}$. %Similar notation for index sets are used in the rest of our paper. 
Our numerical examples will be restricted to $d=1$ or $d=2$ on a closed interval or a closed rectangle $\Omega$, respectively, but the theoretical framework we develop holds for general choices of $d$ and $\Omega$. 

We assume throughout this document that $V$ has no common zeros on $\Omega$, i.e., that, %for every $x \in \Omega$,
\begin{align*}
  \forall x \in \Omega \;\; \exists \; v \in V \textrm{ such that } v(x) \neq 0.
\end{align*}
This assumption is true if, for example, $V$ contains constant functions.

\subsection{The Unconstrained Problem -- linear measurements}\label{ssec:unconstrained}
We assume availability of an unconstrained function approximation scheme from $H$ onto $V$ using a finite collection of data. In this section we briefly mention canonical approaches for accomplishing this via linear measurements, but our optimization problem is independent of how this unconstrained approximation is formed.

Let $u \in H$ be a function about which we have a finite number of observations $\{u_m\}_{m\in[M]}\coloneqq \{\phi_m(u)\}_{m\in [M]}\subset \R$,
where $\phi_1,\cdots,\phi_M$ are $M$ linear functionals on $H$, and are bounded on $V$. The functionals can be, e.g., $v_m$-projections $\langle \cdot,v_m\rangle$ or pointwise evaluations $\delta_{x_m}(\cdot)$, where $\delta_{x_m}$ is the Dirac mass centered at $x_m \in \Omega$. An approximation $p \in V$ to $u$ is frequently built by enforcing these linear measurements:
\begin{align}\label{eq:unconstrained}
  \textrm{Find $p= \sum_{n \in [N]} \widehat{p}_n v_n$ satisfying } \bs{A} \bs{\widehat{p}} = \bs{b},
\end{align}
where 
\begin{align}\label{eq:Amatrix}
  (\bA)_{m,n} &= \phi_{m}(v_n), & (m,n) &\in [M] \times [N], & \boldsymbol{b} &= [\phi_1(u),\cdots,\phi_M(u)]^\top\in \mathbb{R}^M.
\end{align}
The condition $M=N$ is necessary for the problem \eqref{eq:unconstrained} to be well-posed, and so in practice one relaxes \eqref{eq:unconstrained} in appropriate ways depending on whether the system is under-/over-determined. For example, with $\bs{\widehat{v}}$ the $v_n$-coordinates of an element $v \in V$, one could relax \eqref{eq:unconstrained} in the following ways:
\begin{equation}\label{eq:unconstrainedrn}
  \begin{aligned}
  (M &> N) & \bhp &= \argmin_{\bhv \in \R^N} \left\| \bs{A} \bhv - \bs{b} \right\|_2 & \textrm{(Least squares)} \\
  (M &= N) & \bhp &= \bs{A}^{-1}\bs{b} & \textrm{(Interpolation)} \\
  (M &< N) & \bhp &= \argmin_{\bhv \in \R^N: \; \bs{A} \bhv = \bs{b}} \|\bs{\widehat{v}}\|_1 & \textrm{(Compressive sampling)}
  \end{aligned}
\end{equation}
where $\|\cdot\|_p$ is the $\ell^p([N])$ norm on vectors. Theory for well-posedness of each of these problems is mature \cite{trefethen_approximation_2012,candes_robust_2006,donoho_compressed_2006,cohen_compressed_2009}. The numerical results in this paper utilize the interpolation ($M= N$) formulation above for simplicity, but this choice is independent of the theory and algorithms developed in this paper. The essential idea is that we assume the ability to construct $\bhp$ that, in the absence of linear inequality or quadratic equality constraints, is considered a good approximation to the original function $u$ based on available data.

\subsection{The Constraints}
In many practical situations, we require not only a solution to \eqref{eq:unconstrained}, but instead a solution that also obeys certain physical constraints, such as positivity over $\Omega$. The unconstrained approximation \eqref{eq:unconstrained} need not obey any such constraints, even if the original function $u$ does obey them, which may lead to unphysical approximations. We therefore consider the problem of imposing these additional constraints. We consider simultaneously imposing two types of constraints: a (possibly uncountable) set of linear constraints, along with a single quadratic constraint.

\subsubsection{The Linear Constraints}
The constraints we consider in this section are motivated by the following examples of structural desiderata:
\begin{itemize}[itemsep=-3pt]
    \item positivity: $p(y)\ge 0$ for all $y\in \Omega$,
    \item monotonicity: $p'(y)\ge 0$ for all $y\in \Omega$,
    \item convexity: $p''(y)\ge 0$ for all $y\in \Omega$.
\end{itemize}
As it is pointed out in \cite{zala2020structure}, these constraints can be characterized by families of linear constraints and a unique solution to the linearly-constrained problem is guaranteed under some mild assumptions. \textit{Linearity} in this context refers to linearity of the constraint with respect to the function $p$ in $V$.
%the orthonormal basis $\{v_n\}_{n\in[N]}$. 
In the rest of this subsection, we briefly review the abstract formulation introduced in \cite{zala2020structure}. 

Assume that there are $K$ types of linear constraints (e.g., $K=2$ if we simultaneously impose positivity and monotonicity). For each $k \in [K]$, each type of linear constraint is a family defined by the condition, %$(L_k, \omega_k)$ and the constraint is,
%, where $L_k(\cdot, y)$ is a $y$-parameterized linear functional on $V$, and $\omega_k$ is a subset of $\Omega$. The linear constraints are,
\begin{align}\label{eq:lconstraintsabs}
  L_k(v,y)&\le 0,&\forall y&\in \omega_k, %\;k \in [K],
\end{align}
where,
\begin{itemize}[itemsep=-3pt]
    \item $\omega_k$ is a subset of the spatial domain $\Omega$,
    \item $L_k(\cdot,y)$ is a $y$-parameterized \textit{unit-norm} element in the dual space $V^*$.
\end{itemize}
The feasible set of elements $v \in V$ that satisfy \eqref{eq:lconstraintsabs} for the family-$k$ constraint is given by
\begin{align}\label{eq:lconstraints}
    E_k \coloneqq \{v\in V\;\;|\;\;L_k(v,y)\le 0, \;\;\forall y\in \omega_k\}.
\end{align}
Positivity, monotonicity, and convexity can be describes by the abstract formulation \eqref{eq:lconstraintsabs}. The \textit{linear feasible set} $E^0$ is the set of all $v \in V$ that satisfy all $K$ constraints simultaneously, and hence is the intersection of all the $E_k$,
\begin{equation}\label{eq:unionlconstraints}
    E^0 \coloneqq \bigcap_{k\in[K]} E_k = \bigcap_{k\in[K]}\{v\in V\;\;|\;\;L_k(v,y)\le 0, \;\;\forall y\in \omega_k\}.
\end{equation}
Note that $E^0$ is always non-empty since it contains $0$.
%We assume the set $E^0$ is non-empty, i.e., that the imposed linear constraints must be consistent. Our algorithms do not provide an approach to check the potential inconsistency.

Since $V$ is $N$-dimensional, we can identify the feasible set $E^0$ in $V$ with a feasible set in the $v_n$-coordinate space $\R^N$.
By the Riesz representation theorem, for any $L\in V^*$, there exists a unique Riesz representor $\ell\in V$ such that,
\begin{align*}
  L(v) &= \langle v,\ell\rangle_H,& \forall v&\in V.
\end{align*}
The function $\ell\in V$ can also be written explicitly using the orthonormal basis $\{v_n\}_{n\in[N]}$,
\begin{align*}
  \ell(\cdot) &= \sum_{n=1}^N\widehat{\ell}_nv_n(\cdot), & \widehat{\ell}_n &= \langle \ell, v_n\rangle = L(v_n),
\end{align*}
and the following relation holds,
\begin{align*}
    &\Vert L\Vert_{V^{*}} = \Vert \ell\Vert_V = \Vert \bhl\Vert,&&\bhl = (\widehat{\ell}_1,\cdots,\widehat{\ell}_N)^{\mathrm{T}}.
\end{align*}
 In what follows we denote the Riesz representor for $L_k(\cdot, y)$ by $\ell_k(\cdot,y)$ and the corresponding coordinate vector by $\boldsymbol{\widehat{\ell}}_k(y)\in\R^N$. Since $L_k(\cdot, y)$ is unit-norm, we have
 \begin{align}
     &\Vert L_k(\cdot, y)\Vert_{V^{*}} = \Vert \ell_k(y)\Vert_V = \Vert \bhl_k(y)\Vert=1.
 \end{align}
Finally, the set $C_k\subseteq\R^N$ corresponding to the feasible set $E_k \subseteq V$ is given by,
\begin{align}\label{eq:lkconstraintsrn}
    %&C_k = \bigcap_{y\in\omega_k}\left\{\bhv=(\widehat{v}_1,\cdots,\widehat{v}_N)\in \mathbb{R}^N\;\;\middle|\;\; L\left(\sum_{i=1}^N\widehat{v}_iv_i(\cdot), y\right)\le 0\right\}=:\bigcap_{y\in\omega_k}c_k(y),&k \in [K],
  C_k &= \bigcap_{y\in\omega_k}\left\{\bhv\in \mathbb{R}^N\;\;\middle|\;\; \left\langle \bhv, \bs{\widehat{\ell}}_k(y) \right\rangle \leq 0 \right\} \eqqcolon \bigcap_{y\in\omega_k}c_k(y),&k \in [K],
\end{align}
and the set $C^0\subseteq\R^N$ corresponding to $E^0$ is
\begin{align}\label{eq:lconstraintsrn}
    C^0 = \bigcap_{k\in [K]}C_k.
\end{align}
It can be verified that all $C_k, k\in [K]$ are closed, convex cones in $\mathbb{R}^N$ (and that $E_k$ is a closed convex cone in $V$) \cite{zala2020structure} and thus their intersection $C^0$ is also a convex cone. Note that, although $C^0$ is simply a convex cone, the geometry of $C^0$ can be very complicated with infinitely many extreme points since every $C_k$ is the intersection of infinitely many half-spaces $c_k(y)$ if the $\omega_k$ is a set with infinite cardinality, e.g., $\omega_k$ is an interval. 
\begin{remark}\label{rmk:affine}
  In \cite{zala2020structure}, an additional $r_k$ parameter is introduced to define \textit{affine} convex cones as feasible sets. We specialize here to the homogeneous case $r_k = 0$, so that our cones all have vertices at the origin. If $r_k \neq 0$, the problem we consider in this paper is not necessarily well-posed; see Example \ref{ex:nonunique2}.
\end{remark} 
We summarize one example from \cite{zala2020structure} to demonstrate the notation and how it can be specialized to familiar types of constraints.
\begin{example}[Positivity]\label{ex:positivity}
Let $\Omega = [-1,1]$ and $V$ be any $N$-dimensional subspace of $L^2(\Omega)\bigcap L^\infty(\Omega)$. We want to impose a positivity-structure for $v \in V$: $v(x)\ge 0, \forall x\in \Omega$. Thus, only $K = 1$ family is needed and $\omega_1 = \Omega$. Fixing $y \in \omega_1$, the corresponding unit-norm linear operator is given by
\begin{align}\label{eq:positivityconstraint}
  L_1(v,y) &\coloneqq -\lambda(y)v(y), & \lambda(y) &= \left(\sum_{n=1}^Nv_n(y)^2\right)^{-\frac{1}{2}},
\end{align}
  where $\lambda(y)$ is a $y$-dependent normalization factor. The corresponding $y$-parameterized Riesz representor $\ell_1(\cdot, y)$ and its coordinate vector $\bhl_1(y)$ are, respectively,
\begin{align}\label{eq:positivityriesz}
  \ell_1(\cdot,y) &= -\lambda(y)\sum_{n = 1}^Nv_n(y)v_n(\cdot),& %\lambda(y) = \left(\sum_{n=1}^Nv_n(y)^2\right)^{-\frac{1}{2}},
    \bhl_1(y) &= \left[-\lambda(y)v_1(y),\cdots, -\lambda(y)v_N(y)\right]^\top.
\end{align}
%and its coordinate vector for $\ell_1$ is
%\begin{align}\label{eq:positivitycomponents} 
%    \bhl_1(y) = \left[-\lambda(y)v_1(y),\cdots, -\lambda(y)v_N(y)\right]^\top.
%\end{align}
Once the orthonormal basis is specified, $\bhl_1(y)$ can thus be explicitly identified.
\end{example} 
\subsubsection{Constraints that are ``determining"}\label{ssec:determining}
In order to establish uniqueness of the solution to our quadratic-linear constrained problem, we require an additional condition on the linear constraints $(L_k,\omega_k)$ defining $C^0$. 
\begin{definition}\label{def:determining}
  The set of constraints $(L_k, \omega_k)_{k\in [K]}$ is $V$-\textit{determining} if 
  \begin{align*}
    v \in V \textrm{ and } L_k(v,y) = 0 \;\;\forall\; y \in \omega_k \;\;\forall \;k \in [K] 
    \hskip 10pt \Longrightarrow \hskip 10pt 
    v = 0
  \end{align*}
\end{definition}
We will assume $V$-determining linear constraints, which amounts to a technical assumption about the geometry of the associated $\R^N$-feasible set $C^0$ that we later exploit. The $V$-determining condition precludes certain problem setups, but all the practical situations we consider in this paper are $V$-determining. As a simple example, to enforce positivity for every point in $\Omega$ as in \Cref{ex:positivity} we have that $L_1(v, y)$ is a scaled point evaluation at $y$. Therefore, the $V$-determining condition requires that if $v \in V$ satisfies $v(y) = 0$ for every $y \in \Omega$ then $v = 0$, which is a quite natural condition.

For more intuition, the following lists some additional examples, with $\Omega = [-1,1]$, $K=1$, and $\ell_1$ the normalized point-evaluation operator in Example \ref{ex:positivity},
\begin{itemize}[itemsep=-2pt]
  \item If $V = \mathrm{span} \{x^j \;|\; j = 0, \ldots, N-1\}$ and $|\omega_1| \geq N$, then the linear constraint is $V$-determining
  \item If $V = \mathrm{span} \{x^j \;|\; j = 0, \ldots, N-1\}$ and $|\omega_1| < N$, then the linear constraint is \textit{not} $V$-determining
  \item If $V = \mathrm{span} \{x^j \;|\; j = 1, \ldots, N\}$, and $|\omega_1| \leq N$ with $0 \in \omega_1$, then the linear constraint set is \textit{not} $V$-determining.
  \item If $V = \mathrm{span} \{H(x), 1-H(x)\}$, with $H$ the Heaviside function, and $\omega_1 = [-1, 0.5]$, then the linear constraint set is $V$-determining.
  \item If $V = \mathrm{span} \{H(x), 1-H(x)\}$, with $H$ the Heaviside function, and $\omega_1 = [-1, -0.5]$, then the linear constraint set is \textit{not} $V$-determining.
\end{itemize}
Note that the $V$-determining condition is violated only for specialized cases, e.g., either when $\omega_1$ has finite cardinality less than $N$, or when $V$ contains very special types of functions. In all the numerical examples we consider, the linear constraints are $V$-determining.

%\begin{example}[Monotonicity]
%Let $\Omega = [-1,1]$ and $V$ be any $N$-dimensional subspace of \\
%$L^2(\Omega)\bigcap W^{1,\infty}(\Omega)$, where $W^{1,\infty}(\Omega)$ is the Sobolev space of the functions where both the functions and their derivatives are in $L^\infty(\Omega)$. To impose an (increasing) monotonicity structure for $v \in V$: $v'(x)\ge 0, \forall x\in \Omega$, we again let $K = 1$, $\omega_1 = \Omega$. Fixing $y \in \omega_1$, the corresponding unit-norm linear operator is given by
%\begin{align}
%  L_1(v,y)&\coloneqq -\tau(y)v'(y), & \tau(y) &= \left(\sum_{n=1}^Nv_n'(y)^2\right)^{-\frac{1}{2}},
%\end{align}
%The corresponding Riesz representor and its coordinate vector are, respectively,
%\begin{align}\label{eq:monotonicityriesz}
%  \ell_1(\cdot,y) &= -\tau(y)\sum_{n = 1}^Nv'_n(y)v_n(\cdot), &
%%  \tau(y) = \left(\sum_{n=1}^Nv_n'(y)^2\right)^{-\frac{1}{2}},
%    \bhl_1(y) &= \left[-\tau(y)v'_1(y),\cdots, -\tau(y)v'_N(y)\right]^\top.
%\end{align}
%%and its coordinate vector is
%%\begin{align}\label{eq:monotonicitycomponents}
%%    \bhl_1(y) = \left[-\tau(y)v'_1(y),\cdots, -\tau(y)v'_N(y)\right]^\top.
%%\end{align}
%\end{example} 

\subsubsection{The quadratic energy constraint}\label{sect:e-constraint}
In addition to the linear constraints, we further impose a single quadratic norm constraint analogous to an $L^2$-energy of the function. To be precise, we impose that our constrained solution must have same norm as the unconstrained solution, %A natural choice is the Hilbert space norm,
\begin{align}\label{eq:hnorm-constraint}
    \Vert v\Vert_H = \Vert p\Vert_H,
\end{align}
where $p$ is the unconstrained solution with coordinate vector $\bhp$ from solving \eqref{eq:unconstrainedrn}. The corresponding discretized constraint set in $\R^N$ is a sphere with radius $\|\bhp\|$,
\begin{align}\label{eq:hnorm-constraintrn}
    C^H \coloneqq \{\boldsymbol{\widehat{v}}\in\mathbb{R}^N|\Vert\bhv\Vert_{\mathbb{R}^N} = \Vert\bhp\Vert_{\mathbb{R}^N}\},
\end{align}
where $\bhv$ is the coordinate vector for $v$. 

We can now state the overall procedure we consider in this paper: 
%Our method consists of two steps. 
\begin{enumerate}[1.]
    \item Given data $\{u_m\}_{m \in [M]}$, solve the unconstrained problem \eqref{eq:unconstrainedrn} to obtain the unconstrained solution $\bhp$.
    \item Post-process the unconstrained solution $p$ by solving the constrained problem
    \begin{align}\label{eq:constrainedrn}
    \bd = \argmin_{\bhv\in C}\frac{1}{2}\Vert \bhv - \bhp \Vert^2,
\end{align}
where $C = C^{H}\bigcap C^{0}$. This constrained problem optimizes a quadratic function over a subset $C$ of a sphere in $\R^N$. 
\end{enumerate}
%The unconstrained solution is a least squares problem and has been extensively studied, e.g., \cite{boyd2018introduction}. 
Our focus is on the theory and algorithms for the second step, post-processing the unconstrained solution to obtain a structure-preserving solution. We show in \Cref{thm:uniqueness} that the problem above has a unique solution.
%A unique solution to \eqref{eq:constrainedrn} can be guaranteed by \Cref{prop: sphericalconvexproj}. 
Furthermore, in \Cref{sect:algorithms}, we are able to naturally extend the existing algorithms proposed in \cite{zala2020structure} to the new optimization problem. We end this section with two examples that illustrate why some alternative formulations to the two-step procedure above do not necessarily result in well-posed problems.  
%\begin{remark}
%In our framework, the unconstrained solution is not required to be the least square solution. One can replace the least square solution in step 1 with other numerical approximations to the original function, e.g., interpolants. 
%\end{remark}
\begin{example}[An alternative formulation with nonunique solutions]\label{ex:nonunique}
One possible alternative to our framework proposed above is to instead consider the following constrained problem
    \begin{align}\label{eq:altconstrainedrn}
    \bd = \argmin_{\bhv\in C}\frac{1}{2}\Vert \bA\bhv - \boldsymbol{b} \Vert^2.
\end{align}
  with $\bA$ and $\boldsymbol{b}$ as introduced in \Cref{ssec:unconstrained}. This formulation incorporates the constraints and a least squares problem simultaneously. However, the solution to this alternative formulation \eqref{eq:altconstrainedrn} is \textit{not necessarily unique}. The issue lies in the fact that if the singular values of the full-rank matrix $\bA$ are not all equal to 1, then the problem corresponds to optimization over an ellipsoid, which can yield non-unique solutions.

For example, consider the case when $N = M = 2$. Let 
$$
\bA = \begin{bmatrix}0.4&0\\0&1\end{bmatrix}, \;\;\boldsymbol{b} = \begin{bmatrix}0 \\ 0.5\end{bmatrix},\;\; C = \left\{(\cos t, \sin t)\;\;\middle|\;\;  0.01\le t\le\pi-0.01 \right\}.
$$
The constrained set $C$ is a spherically convex set (\Cref{def:sconvex}). The loss/cost function associated to \eqref{eq:altconstrainedrn} is
$$
cost(t) = 0.4\cos^2 t + (\sin t - 0.5)^2 = 0.6\sin^2 t - \sin t + 0.65,
$$
which has two distinct global minima over the feasible set $C$ at $t = \arcsin(\frac{5}{6})$ and $t=\pi-\arcsin(\frac{5}{6})$.
\end{example}
\begin{example}[Nonhomogeneous cones as in \Cref{rmk:affine}]\label{ex:nonunique2}
In \Cref{rmk:affine} we mention that in previous work \cite{zala2020structure} an $r_k$ parameter is introduced to include more general linear constraints. If $r_k \neq 0$, then our optimization problem does not necessarily have unique solutions. %The reason we do not include it in this work is because it may lead to nonunique solutions.
Consider the following problem. Let the feasible set
\begin{align*}
    C = \left\{(x,y)\middle\vert x^2+y^2 = 1\right\}\bigcap\left\{(x,y)\middle\vert y\le 4x+2, y \le -4x+2 \right\}.
\end{align*}
The feasible set consists of two disjoint arcs (red arcs in \Cref{fig:nonunique}). If the unconstrained solution lies in the middle of the dark green arc (the black dot), there will be two solutions to \eqref{eq:constrainedrn}, one from each red arc.
\begin{figure}[htbp]
\centering
\includegraphics[width=0.5\textwidth]{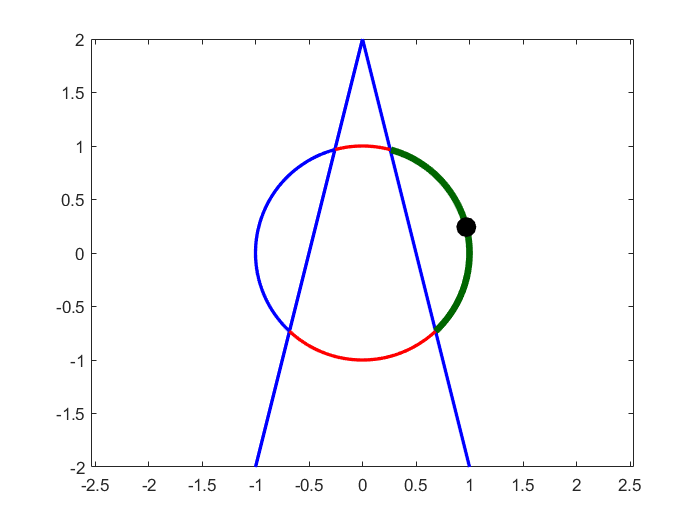}
\caption{An illustration to \Cref{ex:nonunique2}}
\label{fig:nonunique}
\end{figure}
\end{example}
\section{Solution to The Constrained Optimization Problem}\label{sect:soln}
In this section, we will study the solution to the constrained optimization \eqref{eq:constrainedrn}. Specifically, we will show in \Cref{thm:uniqueness} that the solution is unique under reasonable conditions.
\subsection{Spherical geometry }\label{subsect:sphericalgeo}
We introduce some definitions and relevant results for the spherical geometry in this subsection. We refer to \cite{ferreira2013projections,ferreira2014concepts} for technical details. We largely focus on the unit sphere in this section, i.e., $C^{H} = \mathbb{S}^{N-1}$. In section \Cref{sect:algorithms}, we will use the more general origin-centered sphere of nonzero radius. 

%To start, we define intrinsic distance on the sphere. 
The \textit{intrinsic distance} on $\mathbb{S}^{N-1}$ is defined to be the great circle distance between two points, which corresponds to the angle between the two unit vectors in the ambient space.
\begin{definition}[Intrinsic distance on the sphere]\label{def:unit-spdist}
    Given $\boldsymbol{u},\boldsymbol{w}\in\mathbb{S}^{N-1}$, the intrinsic distance between them is 
    \begin{align}\label{eq:unitspheredist}
        d(\boldsymbol{u},\boldsymbol{w}) = \arccos\langle\boldsymbol{u},\boldsymbol{w}\rangle.
    \end{align}
    If $S$ is an origin-centered sphere with radius $r > 0$, then the intrinsic distance between $\boldsymbol{u},\boldsymbol{w}\in S$ is 
    \begin{align}\label{eq:spheredist}
        d_r(\boldsymbol{u},\boldsymbol{w}) = r \arccos\left\langle \boldsymbol{u}/\Vert\boldsymbol{u}\Vert,\boldsymbol{w}/\Vert\boldsymbol{w}\Vert\right\rangle.
    \end{align}
\end{definition}
%\begin{definition}[Intrinsic distance of a unit sphere with radius $r$]\label{def:spdist}
%Let $S$ be a sphere centered at the origin with radius $r$. Given $\boldsymbol{u},\boldsymbol{w}\in S$, the intrinsic distance between them is 
%    \begin{align}\label{eq:spheredist}
%        d_r(\boldsymbol{u},\boldsymbol{w}) = r \arccos\left\langle \boldsymbol{u}/\Vert\boldsymbol{u}\Vert,\boldsymbol{w}/\Vert\boldsymbol{w}\Vert\right\rangle.
%    \end{align}
%\end{definition}
Note that the intrinsic distance between two points on a sphere of radius $r \neq 1$ is given by the intrinsic distance between the unit-normalized points, scaled by the radius.

\begin{definition}[Geodesics on a sphere]
A geodesic on the unit sphere $\mathbb{S}^{N-1}$ is a \textit{great circle}, i.e, the intersection curve of the sphere and a hyperplane in $\mathbb{R}^N$ through the origin. The unique arclength-parameterized geodesic segment from $\bs{u}$ to $\bs{w}$, where $\bs{u},\bs{w}\in\mathbb{S}^{N-1}$ and $\bs{u}\ne\pm\bs{w}$, is given by
\begin{align}\label{eq:geosegmentunit}
%    &\gamma_{\bs{u}\bs{w}}(t) = \left(\cos t - \frac{\langle \bs{u},\bs{w}\rangle\sin t}{\sqrt{1 - \langle \bs{u},\bs{w}\rangle^2}}\right)\bs{u} + \frac{\sin t}{\sqrt{1 - \langle \bs{u},\bs{w}\rangle^2}}\bs{w}, && t\in [0, d(\bs{u},\bs{w})].\\
    &\gamma_{\bs{u}\bs{w}}(t) = \csc d(\bs{u},\bs{w}) \left[ \bs{u} \sin(d(\bs{u},\bs{w}) - t) + \bs{w} \sin t \right], && t\in [0, d(\bs{u},\bs{w})].
\end{align}
  %\an{I think you can replace the above by something like $\csc d(\bs{u},\bs{v}) \left[ \bs{u} \sin(d(\bs{u},\bs{v}) - t) + \bs{w} \sin t \right]$. Please verify and do so.}
  The (non-unique) geodesic segments joining $\bs{u}$ and $-\bs{u}$, starting at $\bs{u}$ with velocity $\bs{v}$ satisfying $\Vert\bs{v}\Vert = 1$ at $\bs{u}$, is given by
\begin{align}\label{eq:unitgeosegments}
    &\gamma_{\bs{u}\{-\bs{u}\}}\coloneqq \cos(t)\bs{u}+\sin(t)\bs{v},& t\in[0,\pi].
\end{align}
For a general sphere $S$ centered at the origin with radius $r$, the geodesic segments can be defined via rescaling \eqref{eq:geosegmentunit} or \eqref{eq:unitgeosegments}.
\end{definition}
\begin{definition}[Exponential Mapping]
The \textit{exponential mapping} at $\bs{u}$ is defined to be 
\begin{align}\label{eq:expmapunit}
    &exp_{\bs{u}}: T_{\bs{u}}\mathbb{S}^{N-1}\to\mathbb{S}^{N-1}, &&\bs{v}\to \bs{u}\cos(\Vert\bs{v}\Vert)+\frac{\bs{v}}{\Vert\bs{v}\Vert}\sin(\Vert\bs{v}\Vert),
\end{align}
  which maps an element on the tangent plane $T_{\bs{u}} \mathbb{S}^{N-1}$ at $\bs{u}$ to the endpoint of the geodesic segment of length $\Vert\bs{v}\Vert$ starting at $\bs{u}$ in the direction of $\bs{v}$. In \eqref{eq:geosegmentunit}, the geodesic segment can be expressed as $\gamma_{\bs{u}\bs{w}}(t)=exp_{\bs{u}}(t\gamma'_{\bs{u}\bs{w}}(0))$. For a general sphere $S$ centered at the origin with radius $r$, the geodesic as well as the exponential mapping can be defined via rescaling \eqref{eq:expmapunit}.
\end{definition}
\begin{definition}[Spherically convex set]\label{def:sconvex}
A subset $C\subseteq\mathbb{S}^{N-1}$ is said to be \textit{spherically convex} if for any $\boldsymbol{s},\boldsymbol{t}\in C$, all the geodesic segments joining $\boldsymbol{s}$ and $\boldsymbol{t}$ are contained in $C$.
\end{definition}
\begin{proposition}[\cite{ferreira2013projections}, Proposition 2]\label{prop: sconvexset}
    Let $C\subseteq\mathbb{S}^{N-1}$. $C$ is a spherically convex set if and only if the cone 
    \begin{align}\label{eq:KC-def}
        K_C = \{z\boldsymbol{s}\;|\;\boldsymbol{s}\in C, z\in[0,+\infty)\}
    \end{align}
    is convex (in Euclidean sense) and pointed, i.e., $K_C\bigcap(-K_C)=\{\boldsymbol{0}\}$.
\end{proposition}

\begin{proposition}\label{prop:chemisphere}
A closed hemisphere is \emph{not} spherically convex.
\end{proposition}
\begin{proof}
  Noticing the existence of the antipodal points on a closed hemisphere, then there is a nontrivial $\bs{v}$ such that $\bs{v}, -\bs{v} \in K_C$. Therefore $K_C \cap (-K_C)$ contains at least one nontrivial point, and \Cref{prop: sconvexset} yields the conclusion.
\end{proof}
One major utility of convex sets on the sphere is the ability to perform projections. %The projection onto a close convex set is defined to be 
\begin{definition}[Spherical projection onto a closed convex set]
Let $C\subset \mathbb{S}^{N-1}$ be a spherically convex, closed set. The projection of $\boldsymbol{z}\in \mathbb{S}^{N-1}$ onto $C$ is defined to be:
\begin{align}
  \mathcal{P}^s_C(\boldsymbol{z}) = \left\{\boldsymbol{t} \in \mathbb{S}^{N-1} \;|\;d(\boldsymbol{t},\boldsymbol{z})\le d(\boldsymbol{t},\boldsymbol{r}), \forall \boldsymbol{r}\in C\right\},
\end{align}
i.e., the nearest intrinsic distance projection. 
\end{definition}
The definition above does not immediately reveal uniqueness or computability for this type of projection, but
%Fortunately, these issues can be resolved by related such spherical projections to Euclidean projections.
the following proposition proved in \cite{ferreira2013projections} shows the relation between spherical projection onto a closed spherically convex set and the Euclidean projection onto the convex cone spanned by the spherical convex set.
\begin{proposition}[\cite{ferreira2013projections}, Proposition 8]\label{prop: sphericalconvexproj}
  Let $C\subseteq\mathbb{S}^{N-1}$ be a spherical convex set. Take $\boldsymbol{z}\in \mathbb{S}^{N-1}$. Let $\boldsymbol{u} = \mathcal{P}_{K_C}(\boldsymbol{z})$, be the Euclidean projection of $\boldsymbol{z}$ onto $K_C$, the latter of which is defined in \eqref{eq:KC-def}. If $\boldsymbol{u}\ne\boldsymbol{0}$, then the spherical projection of $\boldsymbol{z}$ onto $C$ is {\rm unique}, and is given by,
    \begin{align*}
        \mathcal{P}^s_C(\boldsymbol{z}) = \frac{\boldsymbol{u}}{\Vert\boldsymbol{u}\Vert} = exp_{\bs{z}}\bs{v},
    \end{align*}
    where 
    \begin{align*}
        &\bs{v} = \left(-\bs{z}\cot\theta+\frac{\bs{u}}{\Vert\bs{u}\Vert}\csc\theta\right)\theta,&\theta = d(\bs{z},\bs{u}/\Vert\bs{u}\Vert).
    \end{align*}
\end{proposition}
%\an{Could you also write the above in terms of the exponential map? Requires computing the velocity in terms of $\bs{z}, \bs{u}$, which I think is not too difficult.}

\subsection{Uniqueness of the solution to the \eqref{eq:opt-equiv2}}
In this subsection, we present the uniqueness theorem for the solution to \eqref{eq:opt-equiv2} (thus to \eqref{eq:constrainedrn}).% given some conditions. 
To start, we provide an equivalent formulation to \eqref{eq:constrainedrn}.
\begin{lemma}\label{lem: gequiv}
    The constrained optimization problem \eqref{eq:constrainedrn} is equivalent to finding the {\rm spherical projection} of $\bhp$ onto the feasible set $C$,
    \begin{align}\label{eq:opt-equiv2}
        \bd  = \argmin_{\bhv\in C}d(\bhv, \bhp).
    \end{align}
\end{lemma}
\begin{proof}
    The proof is direct,
    \begin{equation}
    \begin{aligned}
        \argmin_{\bhv\in C}\frac{1}{2}\Vert\bhv-\bhp\Vert^2
        &= \argmin_{\bhv\in C} \left(\frac{1}{2}(\Vert \bhv\Vert^2+\Vert\bhp\Vert^2) - \langle \bhv,\bhp\rangle\right)\\
        & = \argmax_{\bhv\in C}\langle \bhv,\bhp\rangle,\\
        & = \argmin_{\bhv\in C}d(\bhv,\bhp).
    \end{aligned}
    \end{equation}
\end{proof}
Using \Cref{lem: gequiv}, we are able to show the following uniqueness theorem.
\begin{theorem}\label{thm:uniqueness}
    Assume the following hold for the constraint set $C = C^H\bigcap C^0$ defined by \eqref{eq:lkconstraintsrn}-\eqref{eq:lconstraintsrn} and \eqref{eq:hnorm-constraintrn}:
    \begin{enumerate}[(a)]
        \item The set of constraints $(L_k,\omega_k)_{k\in[K]}$ are $V$-determining.
        \item The Euclidean projection onto the linearly-constrained set satisfies $\mathcal{P}_{C^0}\widehat{\bs{p}} \ne \boldsymbol{0}$.
      \item The set $C^H$ in \eqref{eq:hnorm-constraintrn} is $\mathbb{S}^{N-1}$, i.e., $\|\bhp\| = 1$.
    \end{enumerate}
    Then the solution to \eqref{eq:opt-equiv2} (or equivalently, \eqref{eq:constrainedrn}) is unique.
\end{theorem}
\begin{proof}
  We first show that $C$ is a closed spherically convex set. A subsequent application of \Cref{prop: sphericalconvexproj} will prove the result. 
  
  Since $c_k(y)$ are closed half-spaces, $C^H\cap c_k(y)$ are closed hemispheres and
  $$
  C = C^H\bigcap C^0 = \bigcap_{y\in\omega_k, k\in [K]}(C^H\cap c_k(y))
  $$ 
  is closed. On the other hand, from \Cref{prop: sconvexset}, the set $C$ is spherically convex if and only if the cone $\KC$ (see also \Cref{prop: sconvexset} for the definition) is convex and pointed. Direct calculation shows that $K_C = C^0$, which has been shown to be a closed convex set in \cite{zala2020structure}.
  Define
    \begin{align}
        W \coloneqq C^0\bigcap\{-C^0\} = \left\{\bs{v} \;\middle|\;\left\langle \bhl_k(y),\bs{v} \right\rangle = 0,\forall y\in \omega_k, k \in [K]\right\}.
    \end{align}
     Take $\bx\in W\subseteq C^0$. By assumption (a) and Definition \ref{def:determining}, the only element $v$ of $V$ satisfying $L_k(v,y) = 0$ for every $y \in \omega_k$ and $k \in [K]$ is $v= 0$. Thus, $W = \{\boldsymbol{0}\}$ and therefore $K_C=C^0$ is pointed.
     By \Cref{prop: sphericalconvexproj} set $C$ is a spherical convex set and the solution to \eqref{eq:opt-equiv2} (or equivalently, \eqref{eq:constrainedrn}) is unique.
\end{proof}
\begin{corollary}
  The conclusions in \Cref{thm:uniqueness} hold with loosening assumption (c) to $\|\bhp\| > 0$.
  %by replacing a general sphere centered at the origin.
\end{corollary}
%\begin{proof}
The proof is direct since all arguments hold unchanged via scaling by $\|\bhp\| > 0$.
%    The proof will be similar since the intrinsic distance between two points is simply the radius multiplied by the angles of the two position vectors.
%\end{proof}

\begin{comment}
    \begin{remark}
    Assumption (b) in \Cref{thm:uniqueness}, from the function approximation's perspective, is equivalent to
    \begin{center}
    \textit{``for some $v\in V$ and $k \in [K]$, $L_k(v,y) = 0$ for all $y\in B\subseteq\omega_k$'' $\Leftrightarrow$ ``$v(y) = 0$''}.
    \end{center}    
        In practice, (b) is a reasonable assumption for certain choices of the subspace $V$. For example, when $V$ is a polynomial spaces of degree up to a finite order and the positivity-preserving constraint is imposed for the whole domain $\Omega$ of $y$, the only polynomial in $V$ that satisfies assumption (b) is the zero polynomial. If, for example, only the monotonicity-preserving constraint is imposed, the solution is not expected to be unique since two polynomials with difference only in the constant terms have the same derivatives. However, (b) is not true for an arbitrary choice of $V$, e.g., $V$ is a finite-dimensional subspace spanned by some wavelets.
    \end{remark}
\end{comment}
\section{Algorithm: Spherical Projections}\label{sect:algorithms}
Having established the well-posedness of the problem \eqref{eq:opt-equiv2}, we proceed to discuss algorithms for solving the problem. In particular, we extend some procedures from \cite{zala2020structure} to the spherical optimization problem \eqref{eq:opt-equiv2}. We will shift our focus back to a general sphere centered at the origin with radius $r \neq 0$, equipped with the intrinsic distance $d_r(\cdot,\cdot)$ (\Cref{eq:spheredist}). Here, $r = \Vert\bhp\Vert \neq 0$ is the norm of the unconstrained solution.

\subsection{Spherical Projection Onto A Closed Hemisphere}
To start, we first compute the spherical projection of a point on the sphere onto a closed hemisphere $c_k(y)\bigcap C^{H}$, which later serves as an ingredient of our main algorithms. Note that \Cref{prop: sphericalconvexproj} is not directly applicable since a closed hemisphere is not spherically convex %(\Cref{prop:chemisphere}). 
The proof in this section is elementary but we provide it in order to make our work self-contained. 

\begin{theorem}\label{thm:gprojectionhalfspace}
  Let $\bhp$ be given, and fix $(k,y)$. If $\bhl_k(y)$ is \emph{not} parallel to $\bhp$, then the spherical projection of $\bhp$ onto the closed hemisphere $C^H \cap c_k(y)$ is unique, i.e., the solution $\bc_k(\bhp;y)$ to 
    \begin{equation}\label{eq:opt3-3}
      \bc_k(\bhp;y) = \argmin_{\bhv \in C^H \cap c_k(y)} d_r\left( \bhv, \bhp \right),
%        \begin{aligned}
%            &\boldsymbol{\bd} = \argmax_{\bhv \in \mathbb{R}^N}\langle \bhv, \bhp\rangle\\
%            \text{s.t.}\;\;&\left\langle \bhl,\bhv\right\rangle \le 0,\\
%            &\Vert \bhl \Vert =  \Vert\bhp\Vert,
%        \end{aligned}
    \end{equation}    
    is unique, and is given by 
    \begin{align}\label{eq:hplaneproj2}
      \bc_k(\bhp;y) = \left\{\begin{array}{rl} \bhp, & \bhp \in C^H \cap c_k(y), \\
        \frac{\mathcal{P}_L \bhp}{\|\mathcal{P}_L \bhp\|} \|\bhp\|, & \bhp \not\in C^H \cap c_k(y),
        \end{array}\right.
    \end{align}
    where $\mathcal{P}_L$ is the Euclidean projection operator onto the subspace $L$, with the latter defined as,
    \begin{align*}
      L \coloneqq \partial c_k(y) = \left\{\boldsymbol{s}\;\middle|\;\left\langle\bhl_k(y),\boldsymbol{s}\right\rangle=0\right\}.
    \end{align*}
\end{theorem}
\begin{proof}
  For simplicity, we will suppress $k,\bhp$, and $y$ notationally in the proof, i.e., $\bc\coloneqq\bc_k(\bhp;y)$, and $\bhl\coloneqq\bhl_k(y)$. Since $C^H \cap c_k(y)$, is non-empty and compact, there is at least one solution to \eqref{eq:opt3-3}. 
  
Following similar computations to the proof to \Cref{lem: gequiv}, it can be shown that
  \begin{align}\label{eq:opt-equiv3}
    \bc = \argmax_{\bhv \in C^H \cap c_k(y)}\langle \bhv,\bhp\rangle,
\end{align}
is equivalent to \eqref{eq:opt3-3}. If $\bhp \in C^H \cap c_k(y)$, then $\bc = \bhp$ is the unique solution to \eqref{eq:opt-equiv3} by the Cauchy-Schwarz inequality, which verifies part of \eqref{eq:hplaneproj2}. Thus, the remainder of the proof assumes $\bhp$ is not in the feasible set. Let $\bhc$ be any solution to \eqref{eq:opt-equiv3}. Since $\bhc$ lies in $c_k(y)$ and since $\bhp$ lies in $C^H$ but is not feasible, then we have
  \begin{align*}
    \left\langle \bhc, \bhl \right \rangle &\leq 0, & \left\langle \bhp, \bhl \right\rangle > 0.
  \end{align*}
  By the above inequalities, any solution $\bhc$ to \eqref{eq:opt-equiv3} satisfies,
  \begin{align*}
    \langle \bhc, \bhp\rangle &= \left\langle \mathcal{P}_L\bhc, \mathcal{P}_L\bhp\right\rangle + \left\langle (I - \mathcal{P}_L) \bhc, (I - \mathcal{P}_L) \bhp \right\rangle = \left\langle \mathcal{P}_L\bhc, \mathcal{P}_L\bhp\right\rangle + \left\langle \bhc, \bhl\right\rangle\left\langle \bhp, \bhl\right\rangle,\\
    &\stackrel{\mathrm{(i)}}{\leq} \left\langle \mathcal{P}_L\bhc, \mathcal{P}_L\bhp\right\rangle 
     \stackrel{\mathrm{(ii)}}{\leq} \|\mathcal{P}_L\bhc\|\,\|\mathcal{P}_L\bhp\| 
     \stackrel{\mathrm{(iii)}}{\leq} \|\bhc\|\,\|\mathcal{P}_L\bhp\| 
     = \|\bhp\|\,\|\mathcal{P}_L\bhp\|
  \end{align*}
  %Note that inequalities (i), (ii), and (iii) above are all equalities if 
  The choice $\bhc = \bc$ in \eqref{eq:hplaneproj2} is the unique solution that achieves equality in (i), (ii), and (iii) above. To see this, first note that $\bc$ is feasible since it lies in both $C^H$ and $c_k(y)$, and is well-defined since $\bhp$ is not parallel to $\bhl_k(y)$ and hence $\mathcal{P}_L \bhp \neq 0$. Equality in (i) and (iii) can be established by noting that $\bc \in L$, so that $\left\langle \bc, \bhl \right\rangle = 0$ and $\mathcal{P}_L \bc = \bc$. Equality in (ii) is achieved if and only if $\mathcal{P}_L \bc = \bc$ has the same direction as $\mathcal{P}_L \bhp$, which the choice \eqref{eq:hplaneproj2} satisfies. This also shows that $\bc$ is the only vector that achieves this equality, and hence \eqref{eq:opt-equiv3} (equivalently, \eqref{eq:opt3-3}) has a unique solution \eqref{eq:hplaneproj2}.
\end{proof}
\begin{remark}\label{rem:rescale}
  The solution \eqref{eq:hplaneproj2} implies that when $\bhp$ is not feasible and is not parallel to $\bhl_k(y)$, the solution to \eqref{eq:opt3-3} can be computed by first computing a Euclidean projection onto the hyperplane $L$, and the simply rescaling this projection to have norm $\|\bhp\|$. We exploit this fact in algorithms.
\end{remark}
\subsection{A Greedy Approach}\label{ssec:greedy}
We first introduce a new notation for the spherical projection
\begin{align}\label{eq:sphereproj}
&\bc_k(\bhp;y) \coloneqq \mathcal{P}^{s}_{c_k(y)}\bhp,
\end{align}
which by Theorem \ref{thm:gprojectionhalfspace} is well-defined for every $\bhp$ that is not a multiple of $\bhl_k(y)$. 

A greedy procedure, in the spirit of the greedy algorithm of \cite{zala2020structure}, iteratively updates $\bhp$ by repeatedly identifying most-violated constraints. Defining $\bhp^0 = \bhp$, and using $\bhp^j$ to denote the iterate at step $j$, we seek to compute,
\begin{align}\label{eq:subspacegreedyproj}
  \bhp^{j+1} &= \mathcal{P}^s_{c_{k^\ast}(y^\ast)} \bhp^{j}, & (k^\ast, y^\ast) &\coloneqq \argmax_{k \in [K], y \in \omega_k} d_r(\bhp^{j}, c_k(y)),
\end{align}
for $j \geq 1$. \Cref{lem:yregion} first allows us to conclude that the set of $(k,y)$ such that $d_r(\bhp, c_k(y) \cap C^H) > 0$ is equal to the set of $(k,y)$ such that $\bhp \not\in c_k(y)$. 
\begin{lemma}\label{lem:yregion}
Let $\bhp$ be the solution to the current iteration, then
\begin{align}
    d_r(\bhp, c_k(y) \cap C^H) > 0 \Leftrightarrow \bhp \not\in c_k(y),
\end{align}
where $r = \Vert\bhp\Vert$.
\end{lemma}
\begin{proof}
    Let $\text{dist}(\cdot,\cdot)$ be the Euclidean distance function, then
    \begin{align}\label{eq:eudist}
      \text{dist}(\bhp, c_k(y)) = \min_{\boldsymbol{s}\in c_k(y)}\Vert\bhp - \boldsymbol{s}\Vert_2=\Vert\bhp - \mathcal{P}_L\bhp\Vert_2 = \Vert\bhp\Vert\sin\theta_k(y),        
    \end{align}
    where $L\coloneqq \partial c_k(y)$ is the boundary of the half-space $c_k(y)$, and $\theta_k(y) = \arccos\left\langle\bhp/r,\bc_k(\bhp;y)/r\right\rangle$ is the angle between $\bhp$ and its spherical projection $\bc_k(\bhp;y)$ onto the half-space $c_k(y)$. The last equality in \eqref{eq:eudist} is true since $\mathcal{P}_L\bhp$ and $\bc_k(\bhp;y)$ have the same direction (\Cref{thm:gprojectionhalfspace}). The angle $\theta_k(y)$ is the angle between the vector $\bhp$ and the plane $L$, thus $\theta_k(y)\in[0,\pi/2]$.
    
    For any $(k,y)$, a direct calculation using the Pythagorean theorem and the definition of the intrinsic distance $d_r(\cdot,\cdot)$ yields
    \begin{align*}
        \bhp \not\in c_k(y)&\Leftrightarrow \text{dist}(\bhp, c_k(y)) > 0,\\
        &\Leftrightarrow \Vert\bhp\Vert\sin\theta_k(y) > 0,\\
        &\Leftrightarrow \theta_k(y) > 0,\\        
        &\Leftrightarrow d_r(\bhp, c_k(y) \cap C^H) > 0,
    \end{align*}
\end{proof}
The proof to \Cref{lem:yregion} motivates the following relation
\begin{equation}\label{eq:mostviolatedcond}
\begin{aligned}
(k^\ast,y^\ast) &= \argmax_{k \in [K], y \in \omega_k}d_r(\bhp, c_k(y))\\
&= \argmax_{k \in [K], y \in \omega_k}(\text{dist}(\bhp, c_k(y)))\\
&=\argmin_{y\in\omega_k, k \in [K]}\text{sdist}(\bhp,c_k(y)),
\end{aligned}
\end{equation}
where the Euclidean signed distance between $\bhp$ and the Euclidean half-space $c_k(y)$ can be computed by (see \cite{zala2020structure})
\begin{align}\label{eq:signedEdist}
    \text{sdist}(\bhp, c_k(y)) = -\langle\bhl_k(y),\bhp\rangle.
\end{align}  
Equations \eqref{eq:mostviolatedcond}-\eqref{eq:signedEdist} imply that, to determine the parameters for the geodesically farthest hemisphere, we only need to determine the parameters for the Euclidean-farthest hyperplane, which is a much easier computational task.

\Cref{alg:greedysphericalproj} summarizes the above iterative procedure of computing the solution to \eqref{eq:opt-equiv2} through greedy spherical projection.

\begin{algorithm}[H]
\setstretch{1.2}
\SetAlgoLined
\KwIn{the matrix $\bA$ and the observational vector $\boldsymbol{b}$, tolerance parameter $\delta \geq 0$}
\KwIn{constraints $\{(\bhl_k(y), \omega_k)\}_{k\in[K]}$}
Compute the unconstrained solution $\bhp$, e.g. via solving \eqref{eq:unconstrainedrn}.\;
  \While{$\text{sdist}(\bhp,c_k(y))\leq-\delta$ for some $k\in[K],\;y\in \omega_k$}{
    compute $(y^*,k^*)$ via \eqref{eq:mostviolatedcond}\;
    update $\bhp$ via \eqref{eq:subspacegreedyproj}.
%\uIf{$\text{sdist}(\bhp,c_k(y))>-\delta$ for all $k\in[K],\;y\in \omega_k$ or \textit{other termination criteria satisfied}}{
%break\;
%}
%\Else{
%update $\bhp$ via \eqref{eq:subspacegreedyproj}.
}
\Return $\bhp$
  \caption{Iterative greedy spherical projection to compute the solution to \eqref{eq:opt3-3}.}
\label{alg:greedysphericalproj}
\end{algorithm}
%In line 5 of \Cref{alg:greedysphericalproj}, $\delta$ is a small positive number, and ``other termination criteria'' are standard optimization termination criteria imposed to avoid infinite loops.%, e.g., setting a maximum number for loops. 
\subsection{An averaging approach}\label{ssec:averaging}
The greedy procedure above can lead to oscillatory behavior of the iteration trajectory. To mitigate this behavior, we introduce an averaging projection approach to suppress potential oscillatory behavior of iterates in the previous greedy approach. The notion of an average position of a collection of points on the sphere is defined by the \textit{Karcher mean}, which is a natural extension of the Euclidean weighted average. 
\begin{definition}\label{def:karcher-mean}[Karcher mean]
  Let $S\subset\mathbbm{R}^{N}$ be the sphere centered at the origin with radius $r$. Let $\bq_1,\cdots, \bq_J$ be $J$ points lying on $S$ associated with nonnegative convex weights $w_1,\cdots, w_J \in [0,1]$. The \textit{Karcher mean} is given by the solution to 
\begin{align}\label{eq:sphericalavg}
    \bq = \argmin_{\bx \in S}\left(\frac{1}{2}\sum_{i = 1}^n w_i\cdot d_r^2(\bx, \bq_i)\right)\coloneqq \argmin_{\bx \in S} f(\bx).
\end{align}
%  where $d_r(\cdot,\cdot)$ is the intrinsic distance on $S$ defined in \eqref{eq:spheredist}, and $\sum_{j=1}^J w_j = 1$.
\end{definition}
The Karcher mean as defined above is unique under mild assumptions.% \cite{buss2001spherical}.
\begin{theorem}[\cite{buss2001spherical}, Theorem 1]\label{thm:savgunique}
  With $S$ the radius-$r$ origin-centered sphere in $\mathbb{R}^N$, suppose that given points $\bq_1,\cdots, \bq_J$ all lie in a closed hemisphere $\mathcal{H} \subset S$, with at least one point $\bq_j$ in the interior of $\mathcal{H}$ with $w_j > 0$. Then $f(\bx)$ defined in \eqref{eq:sphericalavg} has a single critical point $\bq$ in the interior of $\mathcal{H}$, and this point $\bq$ is the global minimum of $f$, hence the unique Karcher mean.
\end{theorem}
%Although the result in \Cref{thm:savgunique} is on the unit sphere, one can easily extend it for any sphere centered at the origin by rescaling.

The definition of the Karcher mean in \eqref{eq:sphericalavg} can be extended to a collection of infinitely many points by integration, and our averaged projection algorithm is based on this generalized Karcher mean. Let $\bhp^0 = \bhp$ be the first iterate in the algorithm. 
%averaged spherical projection (to be defined below) at $j$-th iteration and $\bhp^0 = \bhp$. 
To compute the next iterate, the averaging algorithm first identifies all the parameters $(y,k)$ for which the associated linear constraints are violated,%violated $y$-region for each linear constraint
\begin{align}
    \omega^j_{k^{-}} = \{y \in \omega_k\;|\;\text{sdist}(\bhp^j,c_k(y))<0\}.
\end{align}
Instead of projecting onto the most violated constraint (as in the previous section) we seek a point that minimizes the Karcher mean objective over all violated constraints. With $r = \Vert\bhp\Vert$, the averaged position $\bhp^{j+1}$ at the next iteration is given by
\begin{align}\label{eq:weightedavgprojection}
    \bhp^{j+1} =  \argmin_{\bx\in C^H}\left(\frac{1}{2}\sum_{k\in \mathcal{D}_j^{-}}\int_{\omega^j_{k^{-}}}
    d^2_r(\bx, \bc_k(\bhp^j;y))
    w_k(\bhp^j, y) %\left(\frac{d_r(\bhp^j,\bc_k(\bhp^j;y))}{K\int_{\omega^j_{k^{-}}}d_r(\bhp^j,\bc_k(\bhp^j;z))dz}\right)
    dy\right),
\end{align}
where $\mathcal{D}_j^{-} = \{k\in[K]\;\;\vert\;\;\omega^j_{k^-}\ne\varnothing\}$ is the set of the indexes where the corresponding constraints are violated at $j$th iteration, and the weight, 
\begin{align}\label{eq:weight}
  w_k(\bhp^j; y) \coloneqq \left(\frac{d_r^2(\bhp^j,\bc_k(\bhp^j;y))}{\sum_{\ell \in \mathcal{D}_j^-} \int_{\omega^j_{\ell^{-}}}d_r^2(\bhp^j,\bc_\ell(\bhp^j;z))dz}\right)
\end{align}
is introduced for each spherical projection $\bc_k(y)$ in order to prioritize updates that mitigate the impact of the more violated constrained sets. In practice, we approximate the integral \eqref{eq:weightedavgprojection} via quadrature with positive weights, i.e., %Instead, we first discretize the integral, e.g., via quadrature rules on $H = L^2(\Omega)$, and then compute the discrete version of \eqref{eq:weightedavgprojection} 
\begin{align}\label{eq:discreteweightedavgprojection}
    \bhp^{j+1} =  \argmin_{\bx\in C^H}\left(\frac{1}{2}\sum_{k\in \mathcal{D}_j^{-}}\sum_{q=1}^{Q_k}w_{k,q}d^2_r(\bx, \bc_k(\bhp^j;y_{k,q}))\right),
\end{align}
where $Q_k$ is the number of the quadrature points associated with the $k$th constraint, and $\{w_{k,q}\}_{k\in[K],q\in[Q_k]}$ are the product of the (positive) quadrature weight with an approximations to the weight \eqref{eq:weight} at the quadrature points $\{y_{k,q}\}_{k\in[K],q\in[Q_k]}$.

All the discussion above is provided in the context of assuming that the hemisphere condition in Theorem \ref{thm:savgunique} holds. The following lemma shows that all the candidate spherical projection $\bc_k(\bhp^j;y_{k,q})$ indeed lie on the same hemisphere.
\begin{lemma}\label{lem:uniquehemisphere}
    The spherical projections $\bc_k(\bhp;y)$ defined by \eqref{eq:hplaneproj2} with $\bhl = \bhl_k(y), k \in [K]$ are always on the hemisphere
    \begin{align}\label{eq:commonhemisphere}
        \mathcal{H} = \{\bx\;|\; \Vert\bx\Vert = \Vert\bhp\Vert, \langle \bhp, \bx \rangle\ge 0\},
    \end{align}
    for any $y$ and $k$.
\end{lemma}
\begin{proof}
    Since $\Vert\bc_k(\bhp;y)\Vert = \Vert\bhp\Vert$, we only need to verify the second condition in \eqref{eq:commonhemisphere}. Using \eqref{eq:hplaneproj2}, a direct computation yields,
    \begin{align*}
        \left\langle \bhp, \bc_k(\bhp;y)\right\rangle = \Vert\bhp\Vert\Vert \mathcal{P}_{c_{k}(y)}\bhp\Vert\ge 0.
    \end{align*}
\end{proof}
All the above is almost sufficient to guarantee that the algorithm described by \eqref{eq:discreteweightedavgprojection} has a unique solution. The last obstacle we have yet to overcome is to ensure that the points $\bc_k(\bhp^j;y_{k,q})$ are uniquely defined.  
%Unfortunately, the result above is not sufficient to establish that our averaging algorithm is meaningful. Recall that we require the spherical projections $\bc_k(\bhp^j;y_{k,q}))$ to exist and to be unique. 
To ensure this, we make the following assumption.
\begin{assumption}\label{assump:notparallel}
    $\bhl_k(y)$ is not parallel to $\bhp^j$ for all $(k,y)$ pairs for $y\in \omega_k^-$.
\end{assumption}
\Cref{assump:notparallel} is necessary to ensure unique existence of the spherical projections $\bc_k(\bhp^j;y_{k,q})$. %On the other hand, the uniqueness of the projection of the current iteration $\bhp^{j}$ is guaranteed by the assumption and \Cref{thm:uniqueness}. 
Although we cannot yet theoretically justify of \Cref{assump:notparallel}, in all of our numerical experiments, \Cref{assump:notparallel} holds. Under this assumption, we can prove uniqueness of the update \eqref{eq:discreteweightedavgprojection}.
\begin{proposition}\label{prop:uniquewavg}
  Under \Cref{assump:notparallel}, the solution $\bhp^{j+1}$ to \eqref{eq:discreteweightedavgprojection} is unique. %and lies in the interior of $\mathcal{H}$.
\end{proposition}
\begin{proof}
Under \Cref{assump:notparallel} and \Cref{lem:uniquehemisphere}, the candidate points $\bc_k(\bhp^j;y_{k,q})$ are all in the interior of $\mathcal{H}$. Therefore, from \Cref{thm:savgunique}, the solution $\bhp^{j+1}$ to \eqref{eq:discreteweightedavgprojection} is unique and furthermore lies in the interior of $\mathcal{H}$.
\end{proof}

Algorithms that compute the average spherical projection by solving the optimization problem \eqref{eq:discreteweightedavgprojection} can be adapted from \cite[Algorithms A1 or A2]{buss2001spherical} or \cite[Equation 10]{krakowski2007computation}. 

\subsection{A Hybrid Approach}\label{ssec:hybrid}
We propose a final algorithm: the averaging algorithm of the previous section results in less oscillatory iterate trajectories, but moves relatively slowly. The algorithm in this section combines the ideas of the greedy and averaging approach. First we denote the greedy update \eqref{eq:subspacegreedyproj} at the $j$th iteration by $\bhp^{j+1}_g$ and the average update \eqref{eq:weightedavgprojection} by $\bhp^{j+1}_a$. The hybrid update we propose moves in the direction of the averaged update $\bhp^{j+1}_a$, but with a distance defined by the greedy update $\bhp^{j+1}_g$. Specifically, at the $j$th iteration,
\begin{enumerate}[i.]
    \item Compute the geodesic projection $\bhp^{j+1}_g$ and the average projection $\bhp^{j+1}_a$ via \eqref{eq:subspacegreedyproj} and \eqref{eq:discreteweightedavgprojection}, respectively.
    \item If $d_r(\bhp^j,\bhp^{j+1}_g)/r < 10^{-6}$, i.e., the most violated constraint is very close to the current update, we simply perform the greedy update, setting $\bhp^{j+1} = \bhp^{j+1}_g$ . %Although mathematically the hybrid update \eqref{eq:hybridupdate} still works, numerically, it stays in the same place when the $d_r(\bhp^j,\bhp^{j+1}_g)/r$ due to truncation error of the sine function. 
    \item Otherwise, we compute $\bhp^{j+1}$ by moving $\bhp^j$ along the unique geodesic from $\bhp^j$ to $\bhp^{j+1}_a$ by a distance given by the intrinsic distance between $\bhp^j$ and $\bhp^{j+1}_g$. Let $\btp^j, \btp^j_a$, and $\btp^j_g$ be the unit-norm versions of $\bhp^j$, $\bhp^{j+1}_a$, and $\bhp^{j+1}_g$, respectively. Then the update we propose is 
    \begin{align}\label{eq:hybridupdate}
        &\btp^{j+1} = %\exp_{\btp^{j}}\left(\frac{d(\btp^j,\btp^{j+1}_g)}{d(\btp^j,\btp^{j+1}_a)}\bs{v}\right),\nonumber\\
        \exp_{\btp^{j}}\left(d(\btp^j,\btp^{j+1}_g)\bs{v}\right),\nonumber\\        
        &\bhp^{j+1} = \| \bhp^j\| \btp^{j+1},
    \end{align}
    where $\bs{v}$ is the unit-speed velocity at the base point $\btp^j$ of the geodesic segment leading to $\btp^{j+1}_a$, i.e., $\bs{v} = \gamma'_{\btp^{j}\btp^{j+1}_a}(0)$.
    %\begin{align*}
    %    &\bv = \frac{1}{\sin d(\btp^j,\btp^{j+1}_a)}\btp^{j+1}_a - \frac{\cos d(\btp^j,\btp^{j+1}_a)}{\sin d(\btp^j,\btp^{j+1}_a)}\btp^j.
    %\end{align*}
    
\end{enumerate}
We use the greedy update in step 2 above since in this case we are relatively close to the solution, and so typically the greedy procedure converges very quickly.

\section{Numerical Experiments}\label{sect:numeric}
Throughout this section, we take $M=N$ observations, and the observation functionals $\{\phi_n\}_{n\in [N]}$ are chosen to be the projection functionals $\phi_n(\cdot) \coloneqq \left\langle \cdot, v_n \right\rangle$ onto the given subspace $V$. We denote the unknown function by $u$, the $H$-best projection onto $V$ by $v$, the norm-constrained solution (the solution to \eqref{eq:constrainedrn}) by $\vnc$, and the linearly-constrained solution by $\vlc$. I.e., $\vlc$ is the solution to \eqref{eq:constrainedrn} but with only the linear inequality constraints,
\begin{align*}
  \vlc &\coloneqq \sum_{n \in [N]} w_n v_n, & \bs{w} &\coloneqq \argmin_{\widehat{\bs{v}} \in C^0} \frac{1}{2} \left\| \widehat{\bs{v}} - \widehat{\bs{p}} \right\|^2.
\end{align*}
For other choices of observation functionals, e.g., pointwise observations (collocation-based approximations), our theory and algorithms can be generalized naturally. 

For our univariate examples, we consider the Sobolev spaces on a general interval $[a,b]$ as our Hilbert spaces $H$,
\begin{align}
&H^q([a,b])\coloneqq\left\{u:[a,b]\to\R\;\;\middle\vert\;\;\Vert u\Vert^2_{H^q}<\infty\right\},&&\Vert u\Vert^2_{H^q}\coloneqq\sum_{j=0}^q\int_{a}^{b}\left[u^{(j)}(x)\right]dx,
\end{align}
and choose the subspace $V$ according to the choice of the pair $(a,b)$,
\begin{equation}\label{eq:basis}
    \begin{aligned}
        &\text{if }(a,b) = (-1,1),\;&&\text{then }V = V^{\text{poly}}\coloneqq \text{span}\;\left\{\{x^n\}_{n=0}^{N-1}\right\},\\
        &\text{if }(a,b) = (0,\pi),\;&&\text{then }V = V^{\cos}\coloneqq \text{span}\;\left\{\{\cos nx\}_{n=0}^{N-1}\right\}.\\        
    \end{aligned}
\end{equation}
We will test our algorithms for $H^0 (=L^2)$, $H^1$, and $H^2$ using the linear constraint sets,
\begin{itemize}
    \item (Positivity) $U_0 \coloneqq\{u\in H\;\;\vert\;\;u(x)\ge 0\;\; \forall x\in [a,b]\}$
    \item (Monotonicity) $U_1 \coloneqq\{u\in H\;\;\vert\;\;u'(x)\ge 0\;\; \forall x\in [a,b]\}$
    \item (Convexity) $U_2 \coloneqq\{u\in H\;\;\vert\;\;u''(x)\ge 0\;\; \forall x\in [a,b]\}$
\end{itemize}
Although our theoretical result in \Cref{thm:uniqueness} does not guarantee the uniqueness of the solution when a boundedness constraint imposed, we still test our algorithms with imposing the constraint,
\begin{itemize}
    \item (Boundedness) $G_0 \coloneqq\{u\in H\;\;\vert\;\;u(x)\le 1\;\; \forall x\in [a,b]\}$,
\end{itemize}
for some of our tests. When a boundedness constraint is imposed, the hyperplane is an affine plane. In this case, we first project the current iteration to the affine hyperplane, then rescale the point with respect to the vertex $\boldsymbol{r}_0$ of the cone (the projection of the origin onto the affine plane) to the sphere, i.e., the projection \eqref{eq:sphereproj} is replaced by
\begin{align*}
    \bc = \mathcal{P}^s_H\bhp\coloneqq \boldsymbol{r}_0+\frac{\sqrt{\Vert\bhp\Vert^2- \Vert\boldsymbol{r}_0\Vert^2}}{\sqrt{\Vert\mathcal{P}_H\bhp\Vert^2 - \Vert\boldsymbol{r}_0\Vert^2}}(\mathcal{P}_H\bhp - \boldsymbol{r}_0).
\end{align*}

We will also introduce a metric to measure the change between the constrained solutions and the unconstrained solution:
\begin{align}
    \eta_{*} = \frac{\Vert v - v_*\Vert_H}{\Vert v - u\Vert_H},
\end{align}
where asterisk ``$*$'' on the subscript of $v$ can be either ``LC'' (the linearly-constrained approximation using the dissipative formulation in \cite{zala2020structure}) or ``NC'' (the non-dissipative procedure in this article). Since $v-u$ is $H$-orthogonal to $V$, the Pythagorean theorem implies,
$$
\Vert v_* - u\Vert_H = \sqrt{1+\eta^2}\Vert v - u\Vert_H^2.
$$
The quantity $\sqrt{1+\eta^2}$ can therefore be used to measure the error in a constrained solution relative to error in the unconstrained solution (which in this case is the $H$-best approximation from $V$). Values of $\eta$ that are $\mathcal{O}(1)$ indicate that the error committed by the constrained solution is comparable to that of the unconstrained solution. It is also interesting to measure the difference between the linearly-constrained solution and the norm-constrained solution $\Vert \vlc - \vnc\Vert$. In all of our experiments, the norm-constrained solution $\vnc$ differs only slightly from the linearly-constrained solution $\vlc$.

\Cref{alg:greedysphericalproj} is the greedy algorithm, but it is also the template for the average algorithm. To apply the average algorithm, one only needs to replace the update of $\bhp$ with \eqref{eq:discreteweightedavgprojection}. In line 5 of \Cref{alg:greedysphericalproj}, $\delta$ is set to be $10^{-10}$. In addition, we restrict the maximum number of iterations to be $10,000$ to avoid infinite loops.

\subsection{Performance Comparison of Algorithms}
In this section, we present the comparison of solutions from the linearly-constrained optimization in \cite{zala2020structure} and the norm-constrained optimization proposed in our work. We consider the case $(a,b) = [-1,1]$ and degree-$(N-1)$ polynomial approximations with $V = V^{\text{poly}}$. The test functions are chosen as a step function and its antiderivatives:
\begin{align}\label{eq:stepvariants}
    &u_{j+1}(x) \coloneqq c_{j+1}\int_{-1}^{x}u_j(t)dt,
    &&u_0(x) = \left\{
    \begin{aligned}
    &0,&&x\le 0,\\   
    &1,&&x> 0,
    \end{aligned}
    \right.
\end{align}
where $c_{j+1}$ is a normalized constant that ensures the $u_{j+1}(1) = 1$. %The subspace $V = V^{\text{poly}}$. 
We provide a summary of the performance of our three proposed approaches for norm-preserving optimization \Cref{tab:performance}. We observe that, compared to greedy approach, there is a slight decrease in the number of iterations. Both greedy approach and the hybrid approach are much faster than the average approach. The relative error of the three proposed approaches are comparable.
\begin{table}[htbp]
    \centering
        \begin{tabular}{ccccc}    \toprule
            &\multicolumn{2}{c}{$N = 6$}&\multicolumn{2}{c}{$N = 31$}\\
            &$I$&$\eta$&$I$&$\eta$\\\midrule
            Greedy&17&1.1479&23&0.9859\\
            Average&87&1.1483&220&0.9856\\
            Hybrid&15&1.1494&22&0.9885\\
            \hline
         \bottomrule
        \end{tabular}
    \caption{Comparison of number of iterations ($I$) and relative errors ($\eta$)  on the test function $u = u_2$ with positivity-constraint imposed for different values of $N$ and different algorithms. The ambient Hilbert space is $H = L^2([-1,1])$.}
    \label{tab:performance}
\end{table}

\subsection{Polynomial Space Approximation Example}\label{subsect:numeric-comp}
In this section, we continue to consider the approximation using $(a,b) = (-1,1)$ and $V = V^{\text{poly}}$. In our first experiment, we test the capability of our algorithms for approximating the step function $u_{0}(x)$ in \eqref{eq:stepvariants} and the similarity between $\vlc$ and $\vnc$.  We compute the approximation for $N = 6$ and $N = 31$ \Cref{fig:ex1-1-and-2} and consider three choices of linear constraint sets $E^0$ introduced in \eqref{eq:unionlconstraints}: 
\begin{enumerate}[(i)]
    \item (positivity) $E^0 = U_0$,
    \item (positivity and monotonicity) $E^0 = U_0\bigcap U_1$, and
    \item (positivity, monotonicity, boundedness) $E^0 = U_0\bigcap U_1\bigcap G_0$.
\end{enumerate}
We note that, for (i) and (ii), the norm-constrained solutions are simply slight adjustment of the linearly-constrained solution, both visually and quantitatively. The discrepancy is more obvious in case (iii). Increasing the order in the polynomial approximation further decreases the discrepancy between $\vlc$ and $\vnc$. We also note that, the Gibbs'-type oscillations presented in the left column of \Cref{fig:ex1-1-and-2} can be 
alleviated by enforcing the monotonicity and the boundedness constraint. All computed $\eta_{\text{NC}}$ values are order $1$, which shows that our norm-preserving approximations are comparable to the $H$-best approximation. 
 
\begin{figure}[htbp]
    \centering
    \subfloat{\includegraphics[width=0.33\textwidth]{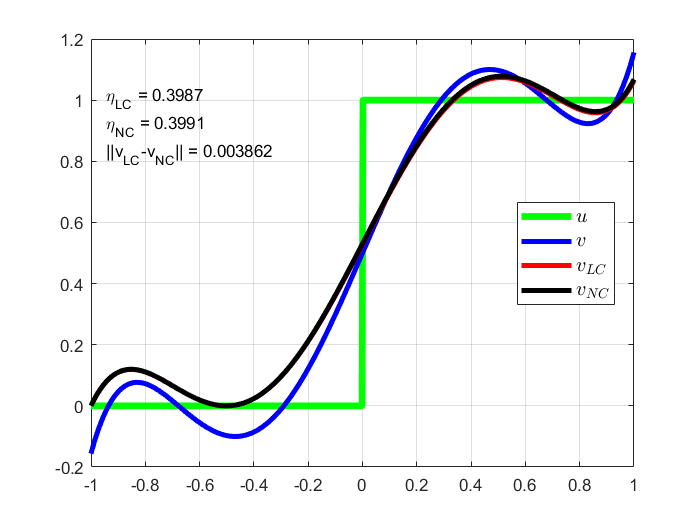}}
    \subfloat{\includegraphics[width=0.33\textwidth]{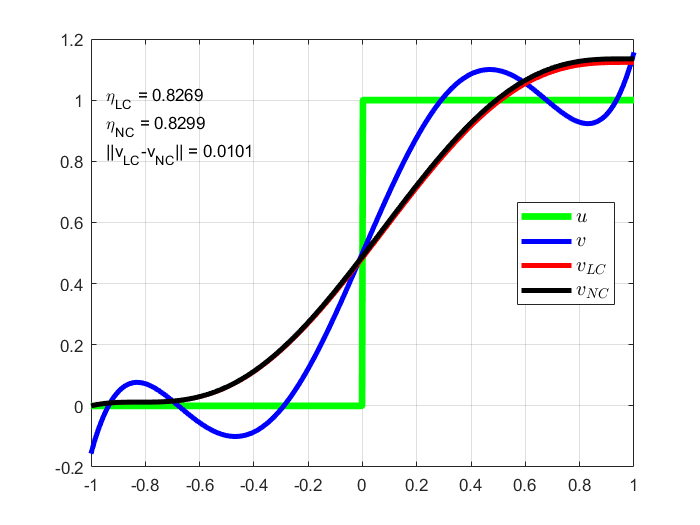}}
    \subfloat{\includegraphics[width=0.33\textwidth]{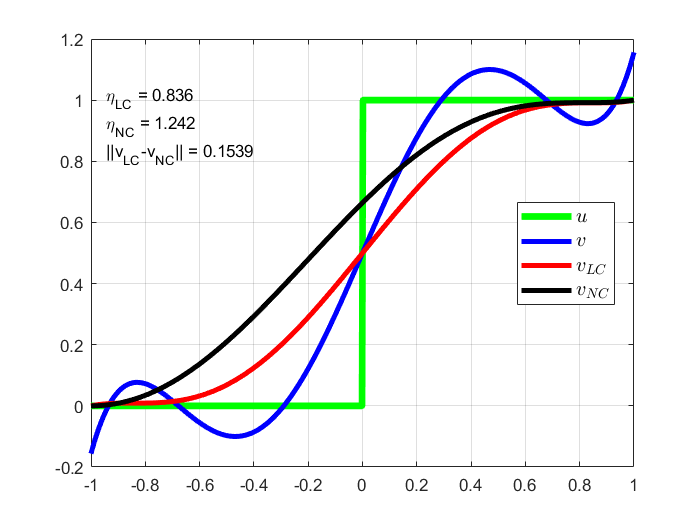}}\\
    \subfloat{\includegraphics[width=0.33\textwidth]{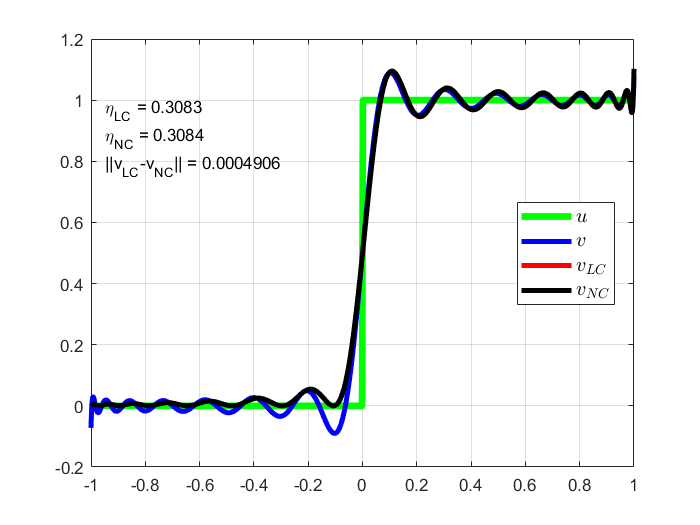}}
    \subfloat{\includegraphics[width=0.33\textwidth]{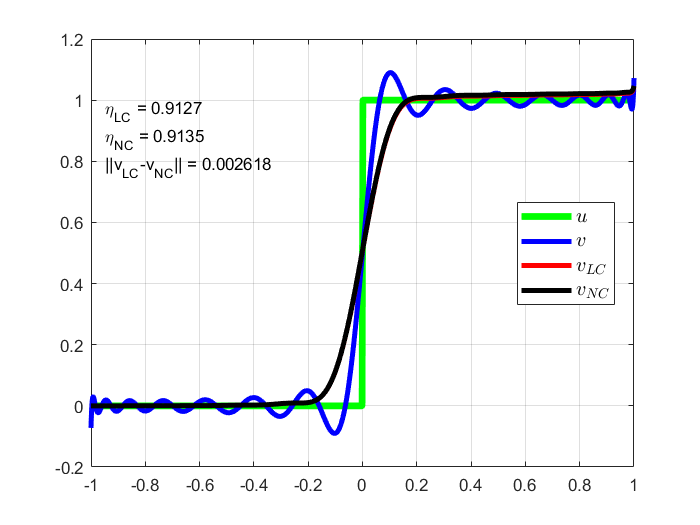}}
    \subfloat{\includegraphics[width=0.33\textwidth]{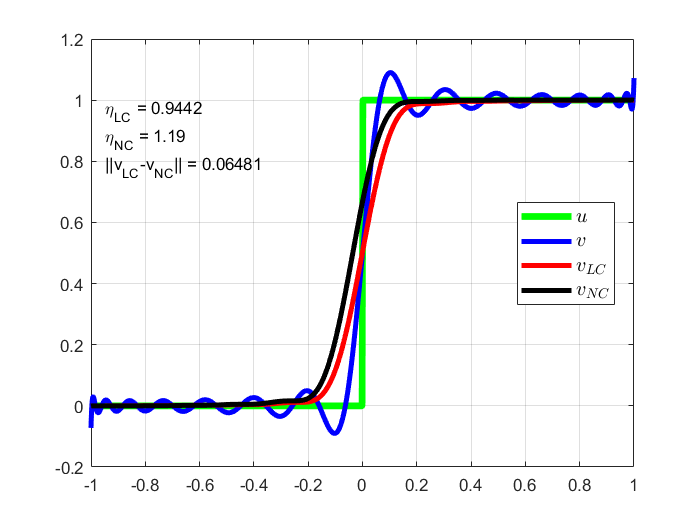}}\\    
    \caption{Greedy algorithm results: Comparison of different methods: degree $5$ polynomial positivity-preserving approximation to the step function for different constraints and different polynomial spaces. Left: constraint $U_0$. Middle: constraint $U_0\bigcap U_1$. Right: $U_0\bigcap U_1\bigcap G_0$. Top: $N = \mathrm{dim} V = 6$. Bottom: $N = \mathrm{dim} V = 31$.}
    \label{fig:ex1-1-and-2}
\end{figure}

In the second experiment of this section, we investigate how the choice of ambient Hilbert space $H$ affects the accuracy of the approximation. We approximate the function $u_2(x)$ with linear constraint set $E^0 = U_0\bigcap U_1\bigcap U_2$ for $N = 6$ and $N=31$ on different Hilbert spaces $H = H^0, H^1,$ and $H^2$. We observe relatively large values of both $\eta_{\text{NC}}$ and $\eta_{\text{LC}}$, but increasing the regularity of the Hilbert space and/or increasing the order of the polynomial can reduce these relative errors.
Similar to the previous test, the discrepancy between $\vlc$ and $\vnc$ decreases as the order of polynomial order increases. It increases as the complexity of the linear constraint set increases. Nevertheless, both approximations are qualitatively good for $N = 31$.

\begin{figure}[htbp]
    \centering
    \subfloat{\includegraphics[width=0.33\textwidth]{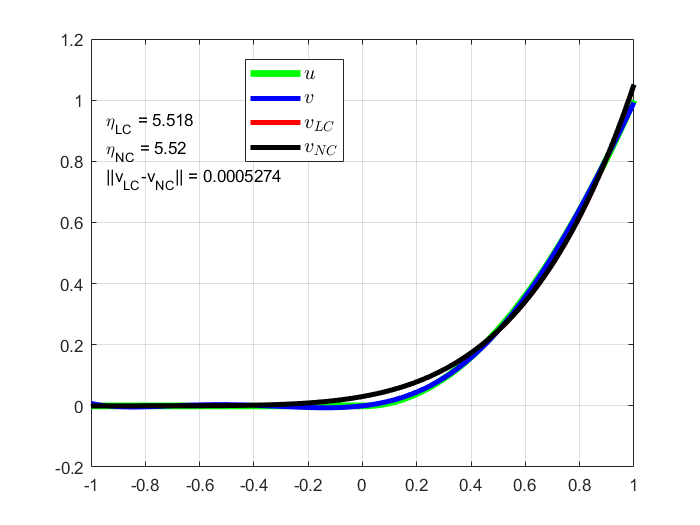}}
    \subfloat{\includegraphics[width=0.33\textwidth]{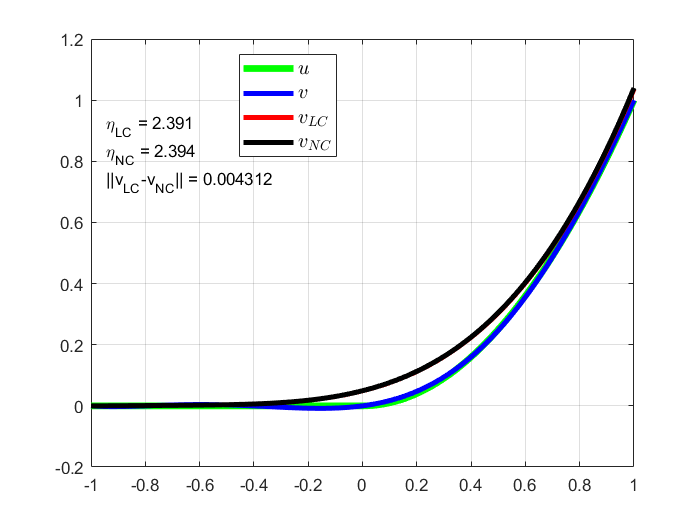}}
    \subfloat{\includegraphics[width=0.33\textwidth]{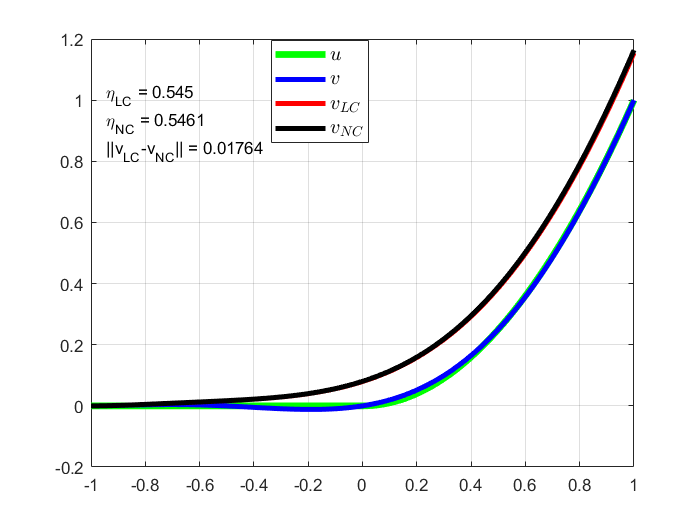}}\\    
    \subfloat{\includegraphics[width=0.33\textwidth]{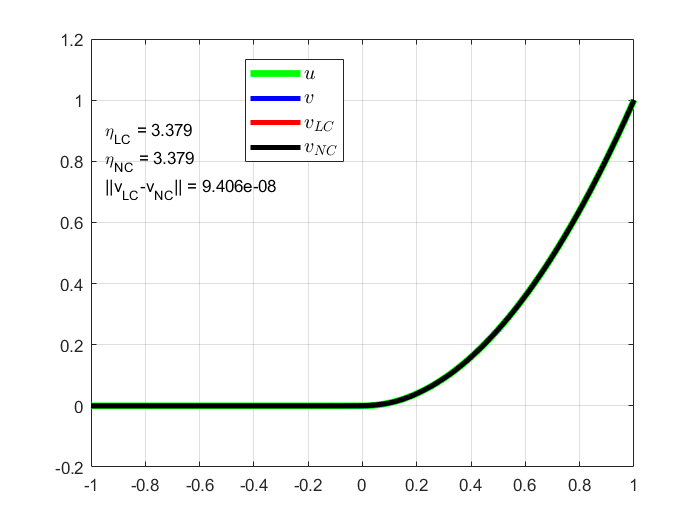}}
    \subfloat{\includegraphics[width=0.33\textwidth]{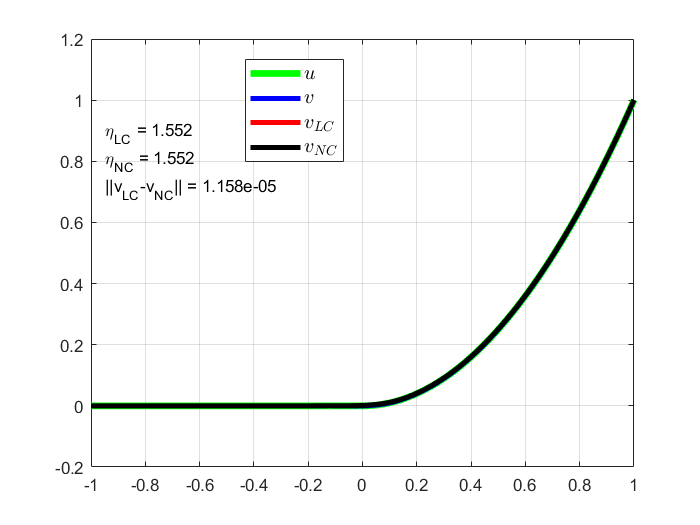}}
    \subfloat{\includegraphics[width=0.33\textwidth]{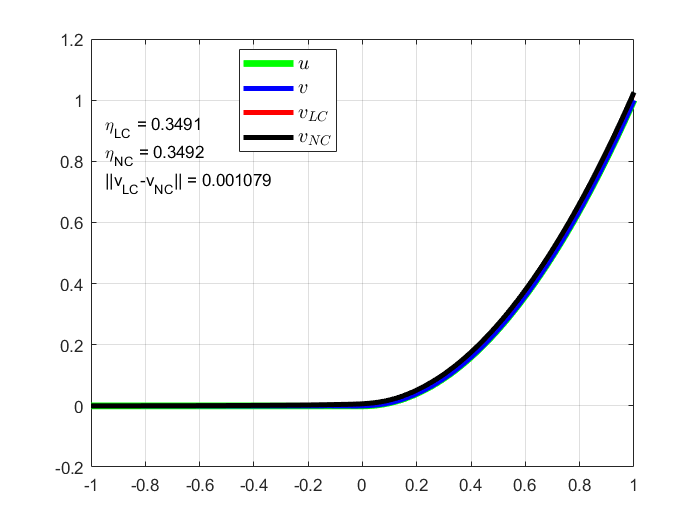}}\\    
    \caption{Comparison of different approximations to $u_2(x)$ for different ambient Hilbert spaces and different polynomial spaces. The red curves ($\vlc$) are covered by the black curves ($\vnc$). See also \Cref{tab:vlc-vnc}. The constraint is $U_0\bigcap U_1\bigcap U_2$. Left: $H = H^0$. Middle: $H = H^1$. Right: $H = H^2$. Top: $N = \mathrm{dim} V = 6$. Bottom: $N = \mathrm{dim} V = 31$.}
    \label{fig:ex2-1-and-2-and-3}
\end{figure}

Quantitatively, we find that the 
minimum values of $v_{NC}, v'_{NC},$ and $v''_{NC}$ converge slowly for some examples with less regular Hilbert space $H = H^0, H^1$. Among our three proposed approach, the greedy approach and the hybrid approach performs slightly better than the average approach. For different choices of the ambient Hilbert spaces, we report in \Cref{tab:min_val} the minimum values of $v_{NC}, v'_{NC},$ and $v''_{NC}$ using average approach at $10,000$ iterations. The ``converge'' in \Cref{tab:min_val} indicates that the procedures achieve the desired tolerance levels. We note that, by increasing the regularity of the ambient Hilbert space, our procedures can identify a feasible solution much faster. We emphasize that the simplicity of this example belies the complexity and difficulty of the geometry of the problem, which is evidenced by algorithms requiring more iterations to complete. In the remaining examples of this paper, all the algorithms identify an element of the feasible set (to within precision tolerances).

\begin{table}[htbp]
    \centering
        \begin{tabular}{ccccccc}    \toprule
            &\multicolumn{3}{c}{$N = 6$}&\multicolumn{3}{c}{$N = 31$}\\
            &$v_{NC}$&$v'_{NC}$&$v''_{NC}$&$v_{NC}$&$v'_{NC}$&$v''_{NC}$\\\midrule
            $H^0$&-1.05e-6&9.35e-5&-1.86e-4&-2.18e-7&-2.35e-3&-3.06e-2\\
            $H^1$&converge&converge&converge&-1.33e-8&-3.84e-7&-1.51e-3\\
            $H^2$&converge&converge&converge&converge&converge&converge\\
            \hline
         \bottomrule
        \end{tabular}
    \caption{Comparison of the minimum values of $v_{NC}, v'_{NC},$ and $v''_{NC}$ using average approach at $10,000$ iterations. ``Converge'' indicates the corresponding procedure finds a feasible solution before the maximum number of iterations is reached.}
    \label{tab:min_val}
\end{table}
\subsection{Constrained Approximation as a Nonlinear Filter}
As pointed out in \cite{zala2020structure}, the original linearly-constrained optimization can be viewed as a nonlinear filter when the linear constraint set $C^0$ contains the origin. Our norm-constrained optimization can also be viewed as a nonlinear filter in the sense that it modifies oscillations to preserve structure. However, we preserve the energy of the approximation, so that it is non-dissipative. In this section, we compare the spectral energy of the approximations $v$, $\vlc$, and $\vnc$ using the examples in \Cref{subsect:numeric-comp}. From \Cref{fig:ex1-1-and-2-bars}, the imposed monotonicity constraint significantly reduces the high-order coefficients in both $\vlc$ and $\vnc$, while the low-order coefficients remain unchanged or are increased. When boundedness constraint imposed, the discrepancy between the coefficients of $\vnc$ and $\vlc$ versus $v$ increases significantly. For most cases in Figure \ref{fig:ex1-1-and-2-bars}, the difference between coefficients of $\vnc$ and $\vlc$ are only slight, which matches the discrepancy between them shown in \Cref{fig:ex1-1-and-2}. From \Cref{fig:ex2-1-and-2-and-3-bars}, we find that the increase in the regularity of the ambient Hilbert spaces will increase the low-order coefficients in both $\vlc$ and $\vnc$. %But the two constrained solutions remain close to each other. 
\begin{figure}[htbp]
    \centering
    \subfloat{\includegraphics[width=0.33\textwidth]{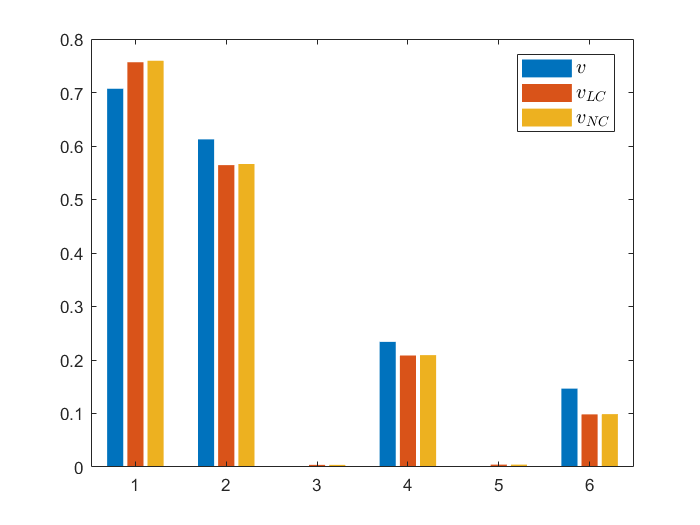}}
    \subfloat{\includegraphics[width=0.33\textwidth]{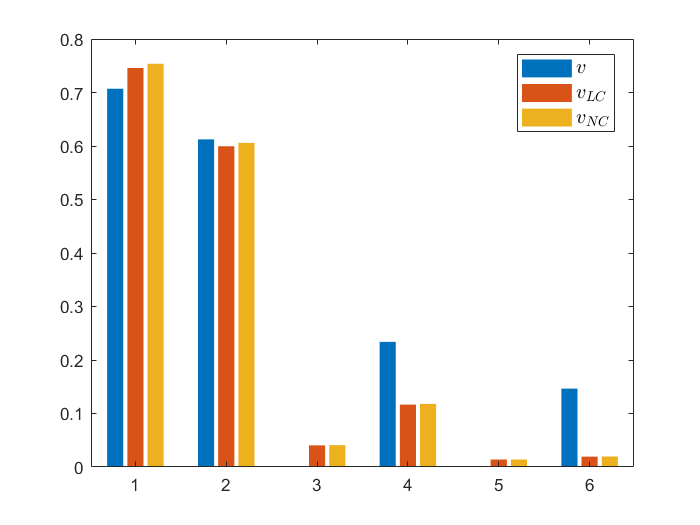}}
    \subfloat{\includegraphics[width=0.33\textwidth]{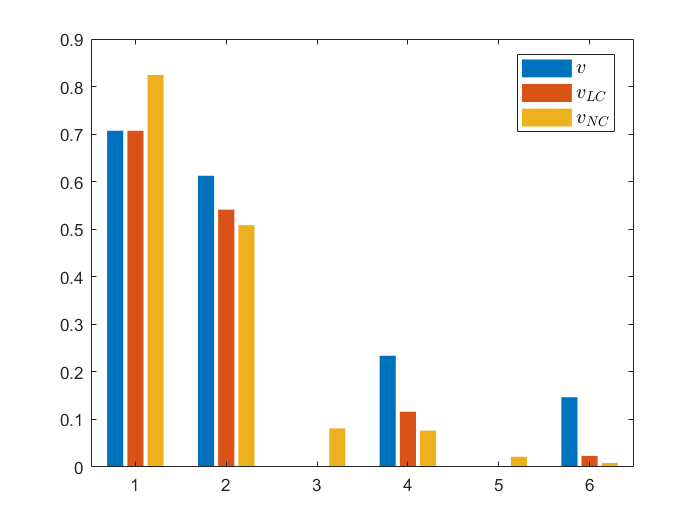}}    
    \\
    \subfloat{\includegraphics[width=0.33\textwidth]{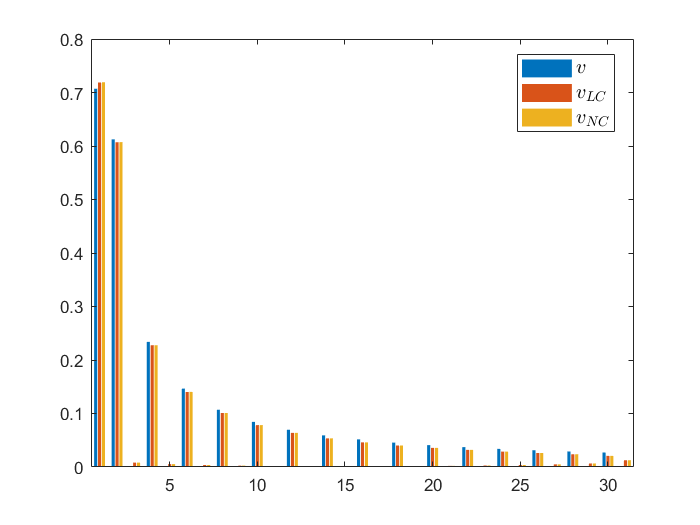}}
    \subfloat{\includegraphics[width=0.33\textwidth]{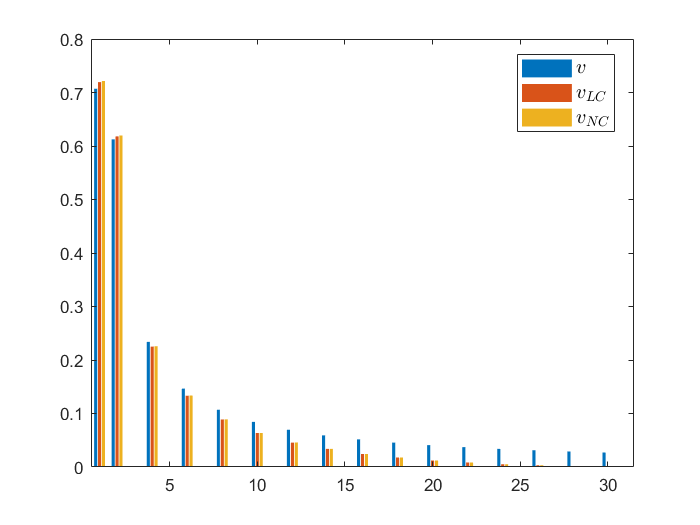}}
    \subfloat{\includegraphics[width=0.33\textwidth]{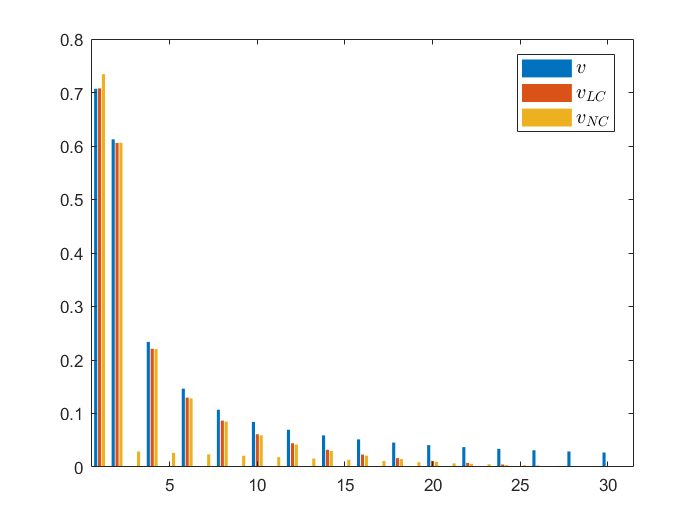}}    
    \\    
    \caption{Companion to \Cref{fig:ex1-1-and-2}. Bar plots showing the magnitudes of the unconstrained projection coefficients $|\widehat{v}_j|$, the linearly-constrained coefficients $|\widehat{\vlc}_j|$, and the norm-constrained coefficients $|\widehat{\vnc}_j|$. The $x$-axis is the index of a coefficient. Left: constraint $U_0$. Middle: constraint $U_0\bigcap U_1$. Right: $U_0\bigcap U_1\bigcap G_0$. Top: $N = \mathrm{dim} V = 6$. Bottom: $N = \mathrm{dim} V = 31$.}
    \label{fig:ex1-1-and-2-bars}
\end{figure}
\begin{figure}[htbp]
    \centering
    \subfloat{\includegraphics[width=0.33\textwidth]{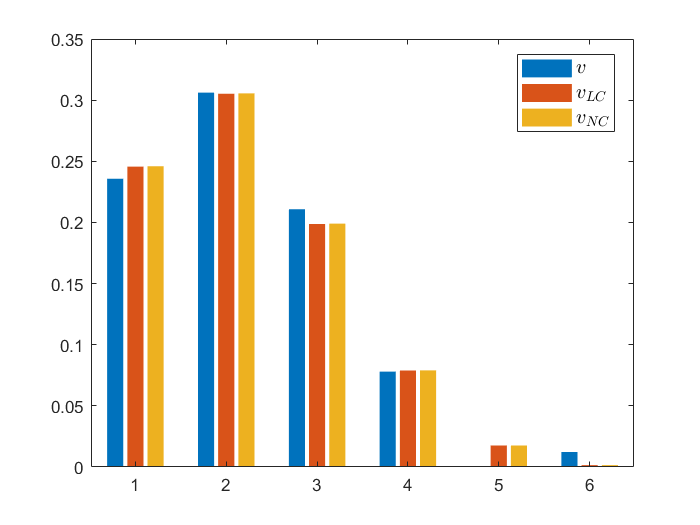}}
    \subfloat{\includegraphics[width=0.33\textwidth]{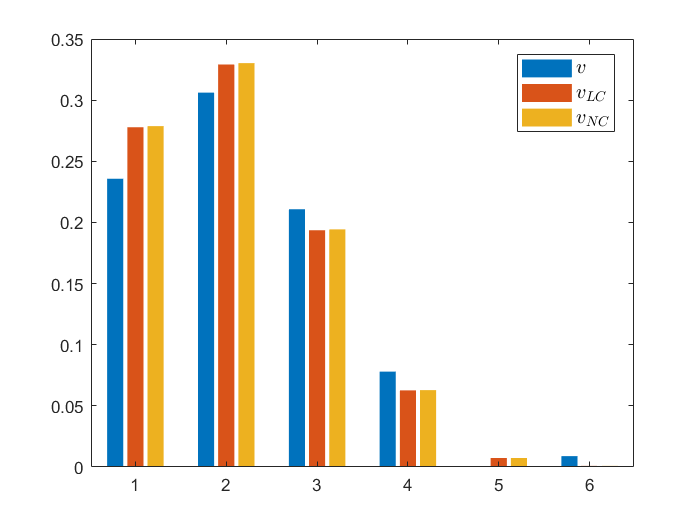}}
    \subfloat{\includegraphics[width=0.33\textwidth]{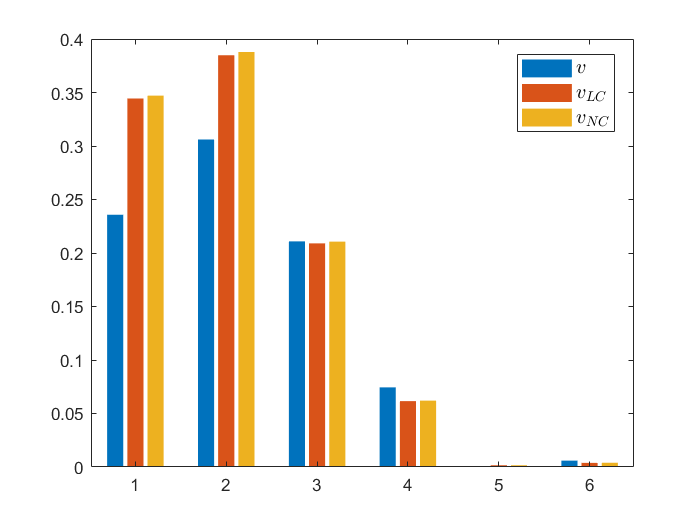}}\\    
    \subfloat{\includegraphics[width=0.33\textwidth]{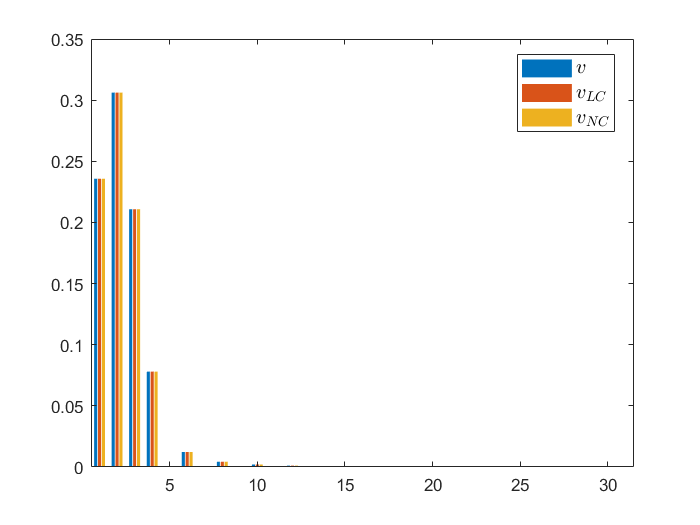}}
    \subfloat{\includegraphics[width=0.33\textwidth]{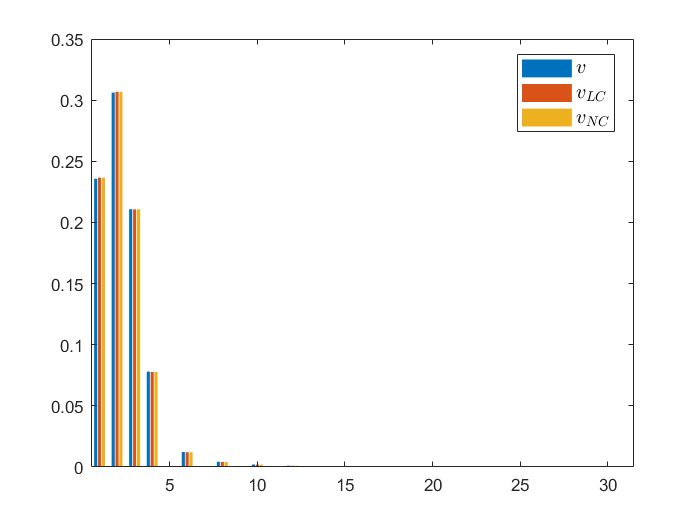}}
    \subfloat{\includegraphics[width=0.33\textwidth]{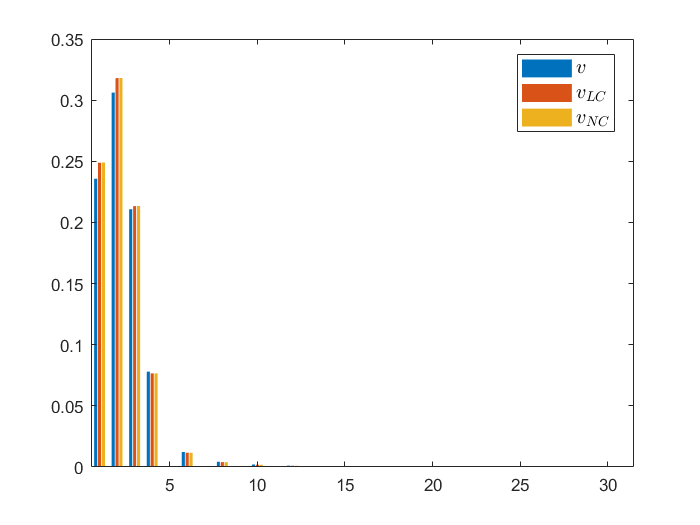}}\\    
    \caption{Companion to \Cref{fig:ex2-1-and-2-and-3}. Bar plots showing the magnitudes of unconstrained projection coefficients $|\widehat{v}_j|$, the linearly-constrained coefficients $|\widehat{\vlc}_j|$, and the norm-constrained coefficients $|\widehat{\vnc}_j|$. The $x$-axis is the index of a coefficient. Left: $H = H^0$. Middle: $H = H^1$. Right: $H = H^2$. Top: $N = \mathrm{dim} V = 6$. Bottom: $N = \mathrm{dim} V = 31$.}
    \label{fig:ex2-1-and-2-and-3-bars}
\end{figure}
\subsection{Convergence Rate}
In this subsection, we compare the rates of convergence between the unconstrained solution and the norm-constrained solution. We consider $u = u_0(x)$ and $u = u_2(x)$ using $V = V^{\text{poly}}$. The ambient Hilbert space is $H = L^2([-1,1])$. We compute the rate of convergence on the constrained sets $U_0$, $U_0\bigcap U_1$, and $U_0\bigcap U_1\bigcap G_0$.
We observe from \Cref{fig:ex1-and-2-conv} that our norm-constrained solutions have a similar rate of convergence to the unconstrained ($H = L^2$-optimal) solution $u$ (even when a boundedness constraint is imposed).
\begin{figure}[htbp]
    \centering
    \subfloat{\includegraphics[width=0.33\textwidth]{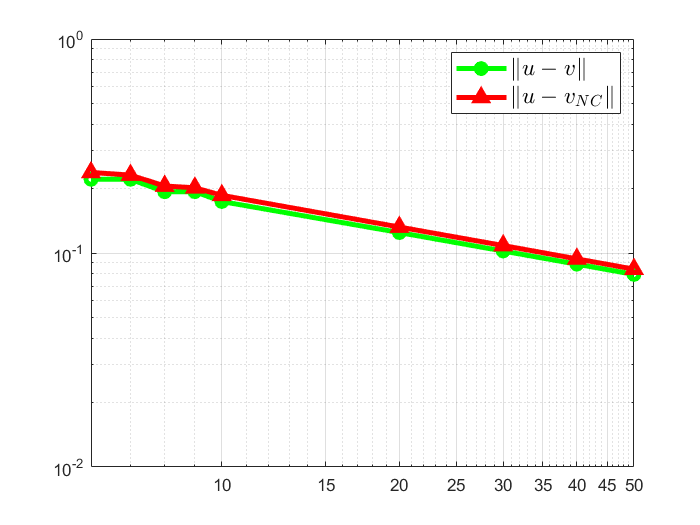}}
    \subfloat{\includegraphics[width=0.33\textwidth]{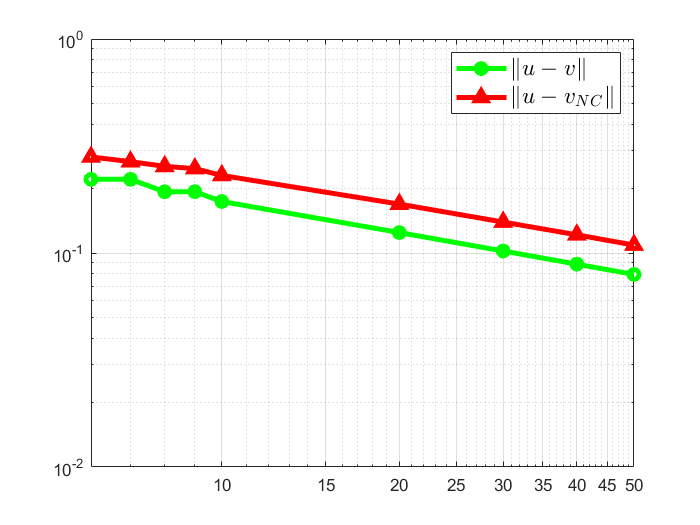}}
    \subfloat{\includegraphics[width=0.33\textwidth]{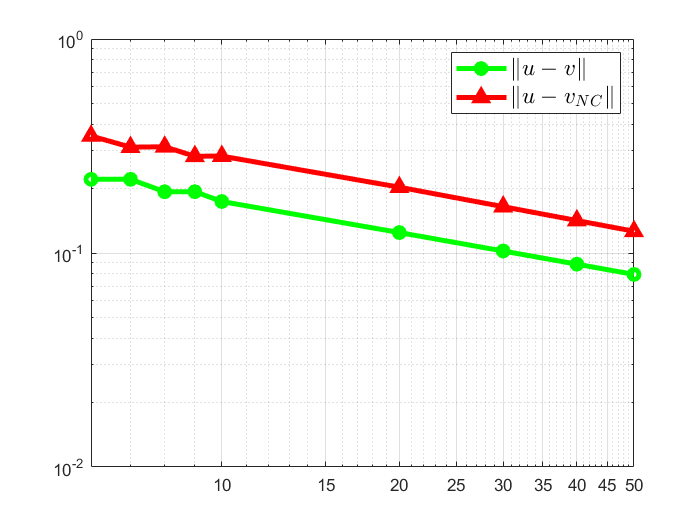}}\\
    \subfloat{\includegraphics[width=0.33\textwidth]{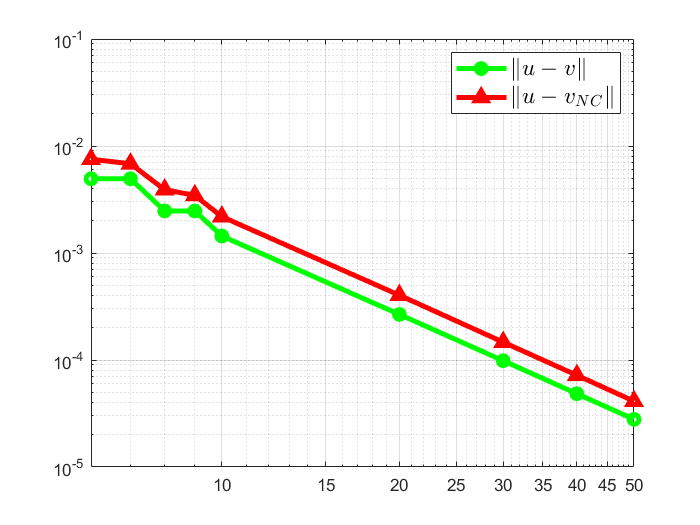}}
    \subfloat{\includegraphics[width=0.33\textwidth]{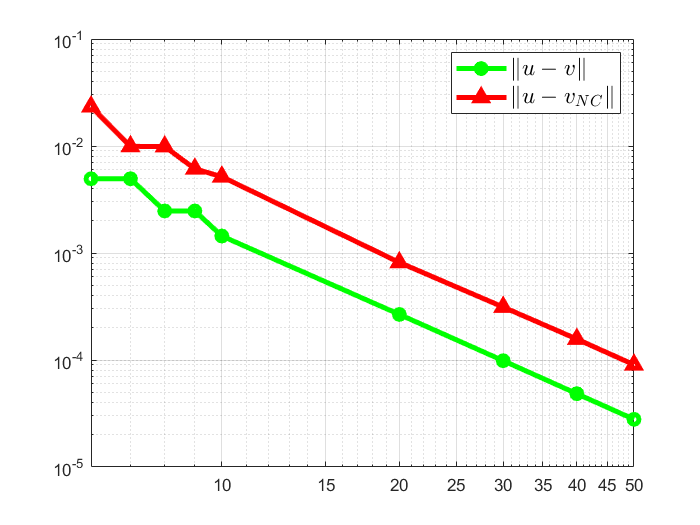}}
    \subfloat{\includegraphics[width=0.33\textwidth]{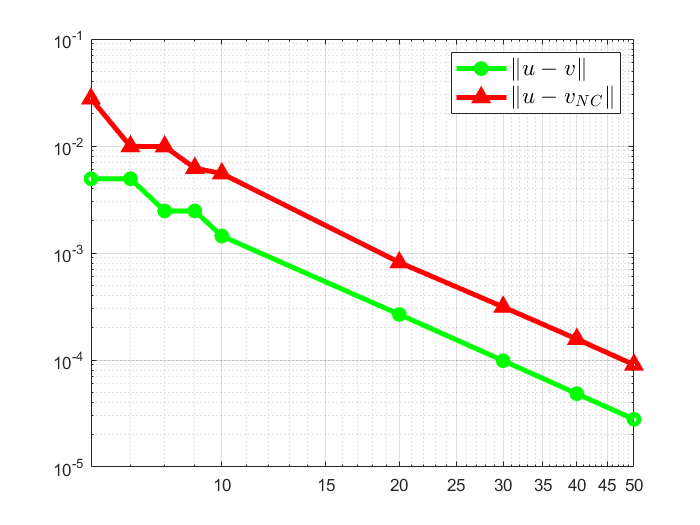}}        
    \caption{Rate of convergence. Approximations to $u = u_0(x)$ (top) and $u = u_2(x)$ (bottom) with $U_0$ (left), $U_0\bigcap U_1$ (middle), and $U_0\bigcap U_1\bigcap G_0$ (right) imposed. The $x$-axis indicates the dimension of the polynomial space $V = V^{\text{poly}}$. The ambient Hilbert space is $H = L^2([-1,1])$.}
    \label{fig:ex1-and-2-conv}
\end{figure}
\subsection{Approximation to A Highly-Oscillatory Function}
In this section, we will approximate the function 
\begin{align}\label{eq:oscifun}
    u(x) = x^2\sin^2\left(\frac{1}{0.01+x^2}\right),
\end{align}
using the polynomial subspace $V = V^{\text{poly}}$ with positivity constraint $U_0$ imposed. We present the results in \Cref{fig:ex3}. From \Cref{fig:ex3}, as the dimension $N$ of $V$ increases, the approximations capture the oscillatory behavior of the original function better. In all subfigures, the red curves ($\vlc$) are covered by the black curves ($\vnc$). In general, the discrepancy between the two solutions remain small and decreases as the order of the polynomial increases.
\begin{figure}[htbp]
    \centering
    \subfloat{\includegraphics[width=0.33\textwidth]{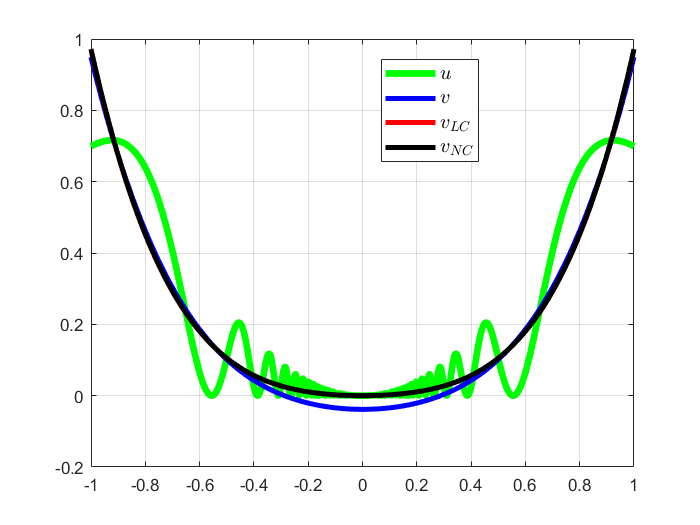}}
    \subfloat{\includegraphics[width=0.33\textwidth]{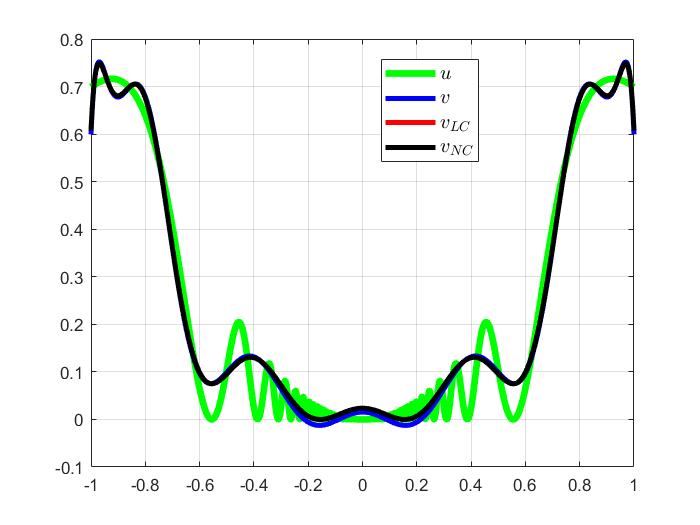}}
    \subfloat{\includegraphics[width=0.33\textwidth]{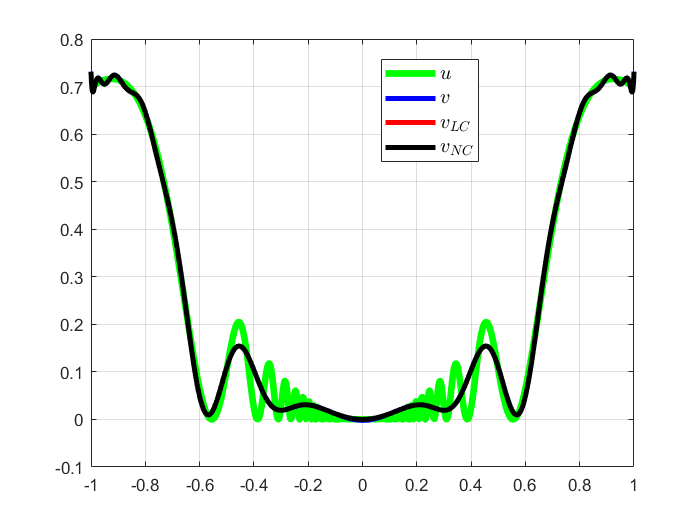}}\\    
    \subfloat{\includegraphics[width=0.33\textwidth]{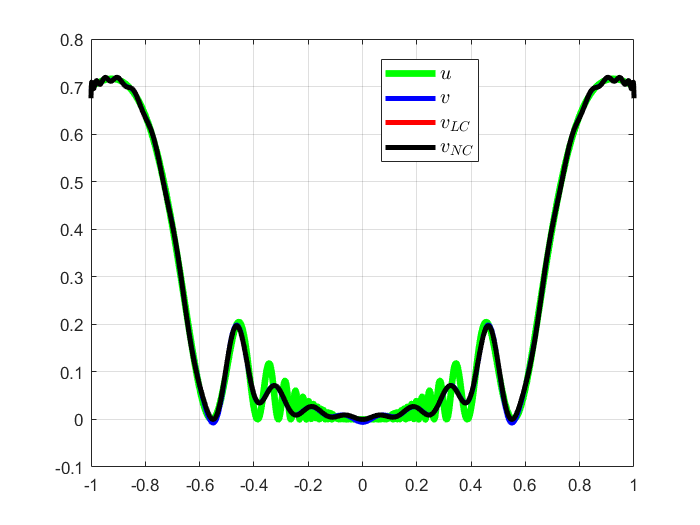}}
    \subfloat{\includegraphics[width=0.33\textwidth]{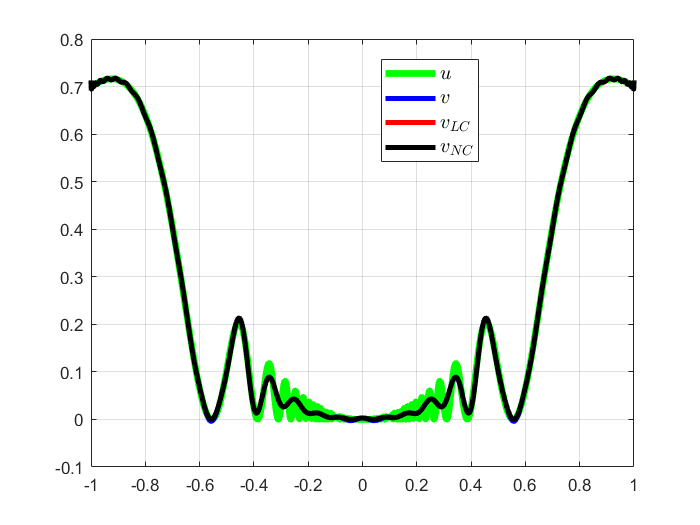}}
    \subfloat{\includegraphics[width=0.33\textwidth]{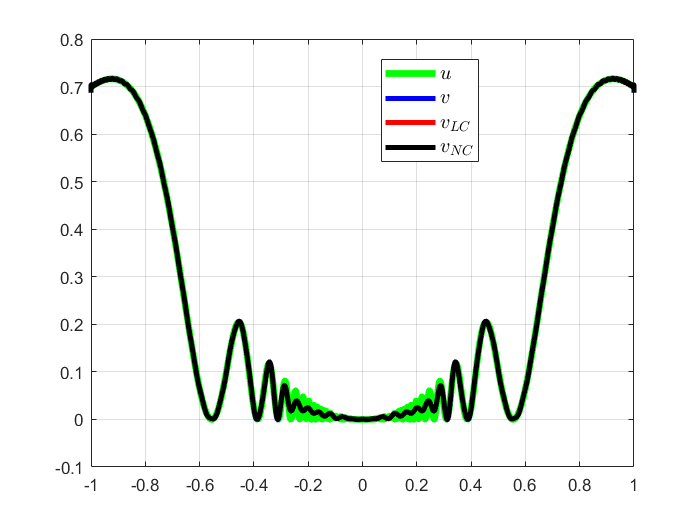}}\\    
    \caption{Comparison of the approximations to \eqref{eq:oscifun} with different $N$. Constraint: $U_0$, positivity-preserving. Top (from left to right): $N = 6, 16, 31$. Bottom (from left to right): $N = 51, 76, 151$. The red curves ($\vlc$) are covered by the black curves ($\vnc$). See also \Cref{tab:vlc-vnc}.}
    \label{fig:ex3}
\end{figure}
\begin{table}[h]
    \centering
    \begin{tabular}{c|c|c|c|c|c|c}
    \toprule
         $N$&  6 & 16 & 31 & 51 & 76 &151\\
         \hline
         $\Vert\vlc-\vnc\Vert$&8.51e-4&7.48e-5&3.38e-7&8.3e-6&1.61e-6&8.22e-8\\
         \bottomrule
    \end{tabular}
    \caption{The decrease in the discrepancy $\Vert\vlc-\vnc\Vert$ with $N$ increasing}
    \label{tab:vlc-vnc}
\end{table}
\subsection{M-shape Function Using Cosine Basis}
In this section, we will choose $V = V^{\cos}$ for approximating an M-shape function defined on $[0,\pi]$,
\begin{equation}\label{eq:M-shape}
    u(x) = \left\{\begin{aligned}
    &-\left(x-\frac{\pi}{8}\right)\left(x-\frac{\pi}{2}\right)&&\frac{\pi}{8}\le x < \frac{\pi}{2},\\
    &-\left(x-\frac{\pi}{2}\right)\left(x-\frac{7\pi}{8}\right)&&\frac{\pi}{2}\le x < \frac{7\pi}{8},\\
    &0&&\text{otherwise},\\
    \end{aligned}\right.
\end{equation}
with positivity constraint $U_0$ imposed. For a cosine polynomial, the difficult part for applying our algorithm is to determine the $y$-parameter corresponding to the most violated constraint (or the negative $y$-region), which requires to find the zeros of a trigonometry polynomials. Fortunately, this difficulty can be resolved by taking advantage of the Chebyshev polynomials.
\begin{figure}[htbp]
    \centering
    \subfloat{\includegraphics[width=0.33\textwidth]{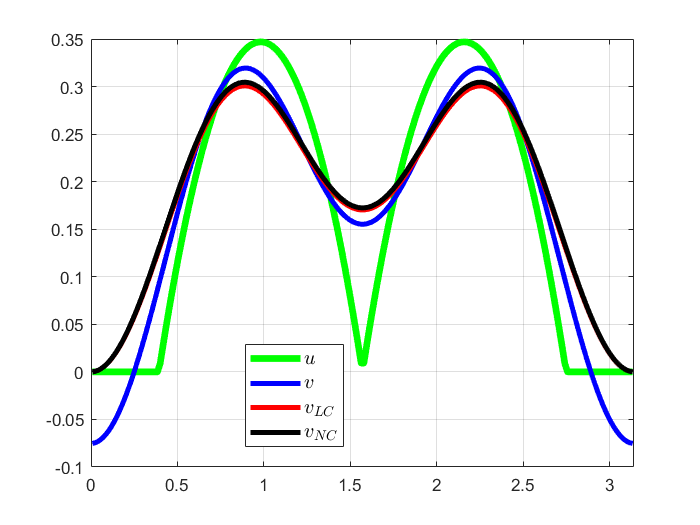}}
    \subfloat{\includegraphics[width=0.33\textwidth]{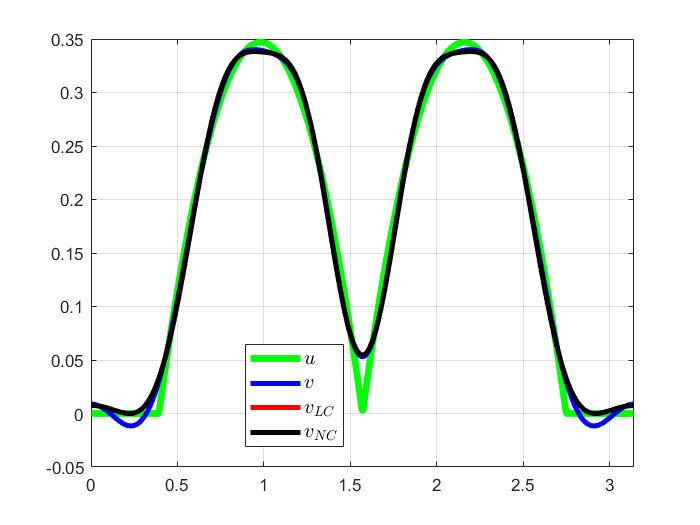}}
    \subfloat{\includegraphics[width=0.33\textwidth]{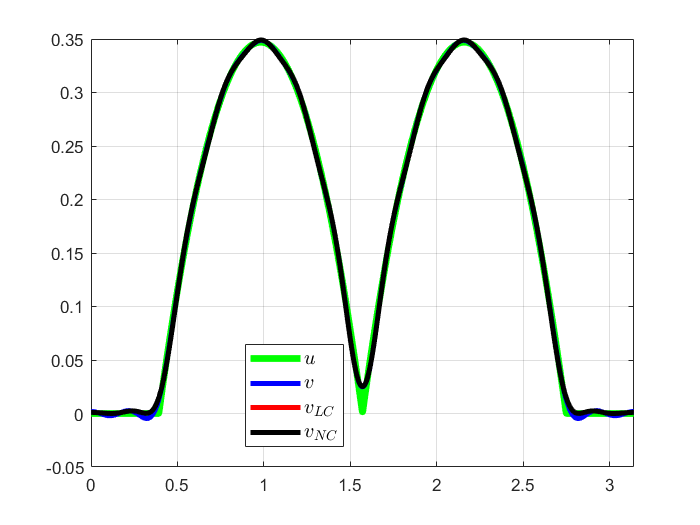}}\\    
    \caption{Comparison of the approximations to \eqref{eq:M-shape} with different $N$. Constraint: $U_0$, positivity-preserving. From left to right: $N = 6, 16, 31$.}
    \label{fig:ex4}
\end{figure}
\subsection{Two-dimensional cylinder indicator function}
In our last example, we consider the approximation to a cylinder 
\begin{equation}\label{eq:cylinder}
u(x,y) = \left\{
    \begin{aligned}
        & 1 && \text{if} \sqrt{(x-0.5)^2 + (y-0.5)^2}< 0.5,\\
        & 0 && \text{otherwise}.
    \end{aligned}\right.
\end{equation}
The computational domain is $[-1,1]\times [-1, 1]$, and the polynomial space is the tensor product space $V^{\text{poly}}\otimes V^{\text{poly}}$, where $N$ is chosen to be $15$. The positivity constraint $U_0$ is imposed. The computation requires to find the global minimum of a two-dimensional nonconvex function \eqref{eq:signedEdist} (see also \eqref{eq:positivityconstraint}-\eqref{eq:positivityriesz}). 
%We generate a \texttt{GlobalSearch} object in \textsc{Matlab} to determine the global minimizer, which computes the global minimum by calling the specified optimization function with random initial initialization multiple times. 
We use \textsc{Matlab}'s optimization function \texttt{fmincon}, using the sequential quadratic programming option, and approximate the global minimum by solving the optimization with several randomly initialized starting points. The constraints we set for \texttt{fmincon} are the boundaries for the computational domain.

The results are shown in \Cref{fig:ex-cylinder}. The numerical results show that our norm-constrained approximation can preserve both the positivity and the norm by ``correcting'' the linearly-constrained solution. The function is entirely non-negative for both the linearly constrained solution (bottom middle, \Cref{fig:ex-cylinder}) and the norm-constrained solution (bottom right, \Cref{fig:ex-cylinder}).
\begin{figure}[htbp]
    \centering
    \subfloat{\includegraphics[width=0.33\textwidth]{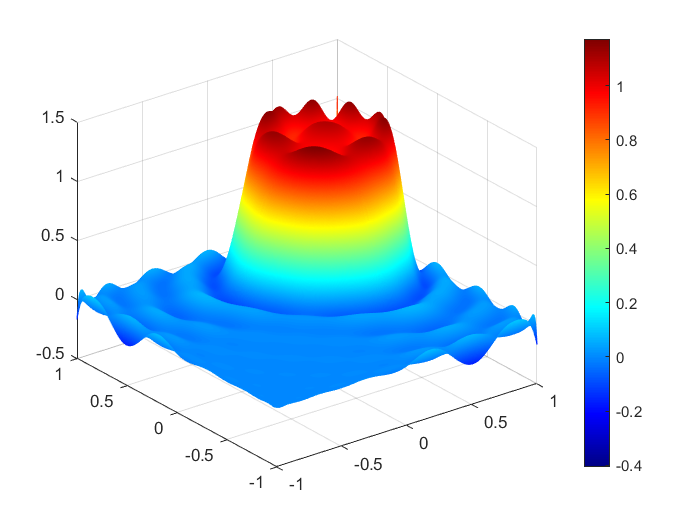}}
    \subfloat{\includegraphics[width=0.33\textwidth]{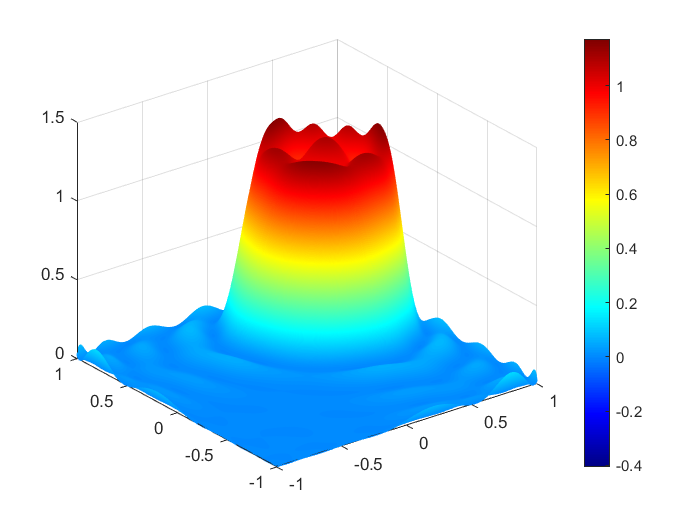}}
    \subfloat{\includegraphics[width=0.33\textwidth]{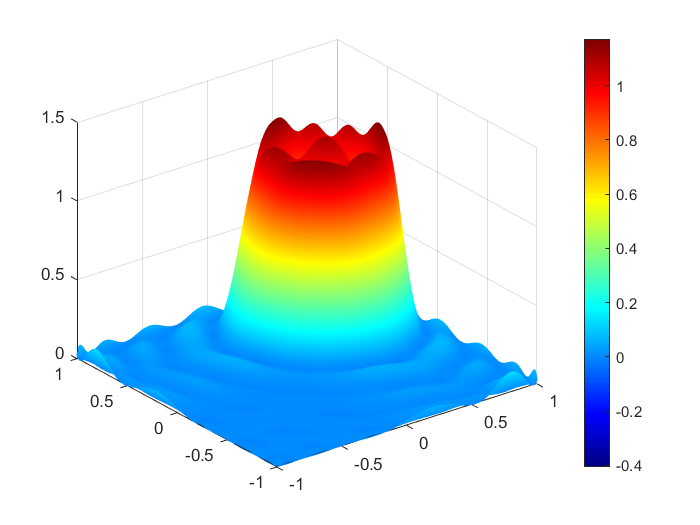}}\\    
    \subfloat{\includegraphics[width=0.33\textwidth]{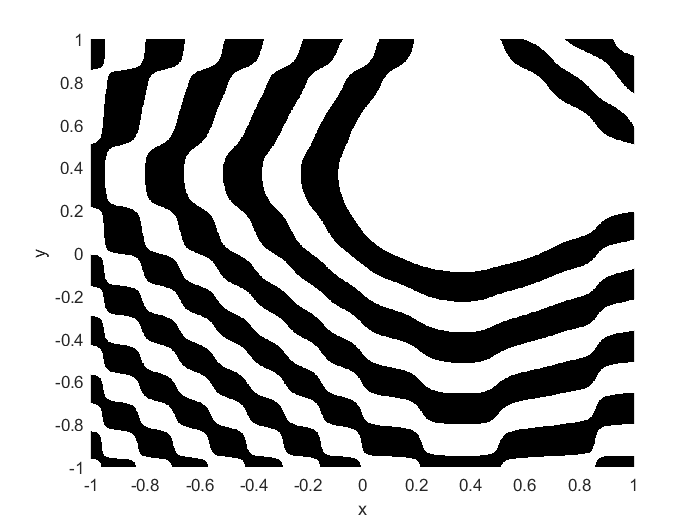}}
    \subfloat{\includegraphics[width=0.33\textwidth]{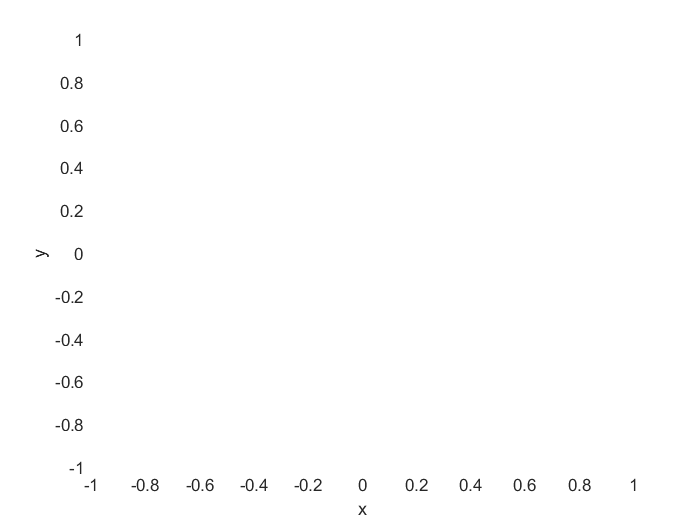}}
    \subfloat{\includegraphics[width=0.33\textwidth]{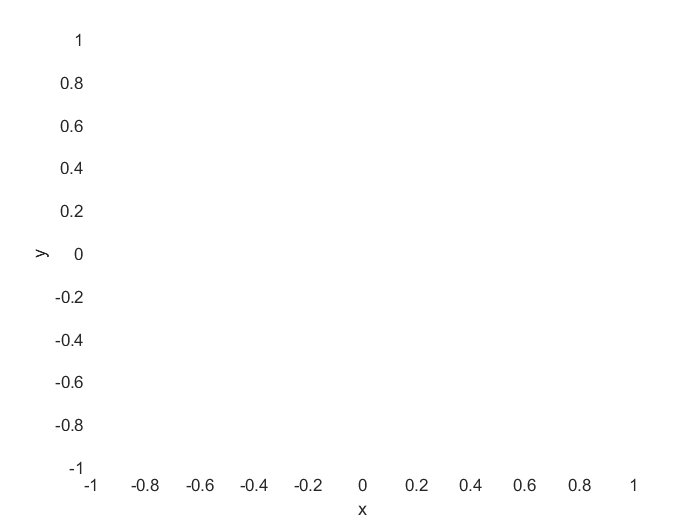}}\\    
    \caption{Comparison of different approximations to \eqref{eq:cylinder}, greedy procedure. Constraint: $U_0$, positivity-preserving. Top: mesh plot. Bottom: negative region indicator function, where the black region represent the region where the approximation is negative. Left: unconstrained solution $u$. Middle: linearly-constrained solution $v_{LC}$. Right: Norm-constrained solution $v_{NC}$. $\eta_{LC} = 0.1229, \eta_{NC} = 0.1230, \Vert v_{LC}-v_{NC}\Vert = 0.0030$.}
    \label{fig:ex-cylinder}
\end{figure}
\newpage
\bibliographystyle{unsrt}
\bibliography{bibfile}
\newpage
\appendix
\listoftables
\listoffigures
\end{document}